\documentclass[11pt]{amsart}

\let\oldmarginpar\marginpar
\renewcommand\marginpar[1]{\-\oldmarginpar[\raggedleft\footnotesize #1]%
{\raggedright\footnotesize #1}}

\usepackage{mathptmx}
\usepackage{amsmath,amssymb,tikz-cd}
\usepackage{quiver}
\usepackage{amsmath}
\usepackage{amssymb}
\usepackage{amsthm}
\usepackage{xspace}
\usepackage[all,tips]{xy}
\usepackage{syntonly}
\usepackage{hyperref}
\usepackage{arydshln}
\usepackage{todo}
\usepackage{enumerate}
\usepackage{graphicx}

\theoremstyle{plain}
\newtheorem{thm}{Theorem}[section]
\newtheorem{cor}[thm]{Corollary}
\newtheorem{prop}[thm]{Proposition}
\newtheorem{lemma}[thm]{Lemma}

\theoremstyle{definition}

\newtheorem{defn}[thm]{Definition}

\DeclareMathOperator{\RF}{RF}
 
\DeclareMathOperator{\Aut}{Aut} \DeclareMathOperator{\Out}{Out}

\DeclareMathOperator{\SL}{SL} 
\DeclareMathOperator{\GL}{GL} \DeclareMathOperator{\PGL}{PGL}

\DeclareMathOperator{\SU}{SU} 
\DeclareMathOperator{\Sp}{Sp}

\DeclareMathOperator{\lcm}{lcm}

\DeclareMathOperator{\D}{D}

\DeclareMathOperator{\Epi}{Epi}

\DeclareMathOperator{\Inn}{Inn}

\newcommand{\Z}{\mathbb{Z}}
\newcommand{\N}{\mathbb{N}}
\newcommand{\Q}{\mathbb{Q}}

\newcommand{\mD}{\mathcal D}

\newcommand{\mO}{\mathcal O}

\newcommand{\mF}{\mathcal F}
\newcommand{\bK}{\mathbb K}

\newcommand{\dc}{\mathrm{dchar}}
\newcommand{\Sym}{\text{Sym}}

\usepackage{fancyhdr}

\title[]{Linearity criteria for automorphism groups of malabelian
groups}
\begin{document}

\author[T. Koberda]{Thomas Koberda}
\address{Thomas Koberda, Department of Mathematics, University of Virginia, Charlottesville, VA 22904}
\email{thomas.koberda@gmail.com}

\author[M. Pengitore]{Mark Pengitore}
\address{Mark Pengitore, Institute of Mathematics of Polish Academy of Sciences, Warsaw, Poland}
\email{mpengitore@impan.pl}

\begin{abstract}
Let $G$ be a finitely generated malabelian group, let $A\leq\Out(G)$
be a finitely generated
subgroup, and let $\Gamma_{G,A}$ denote the preimage of $A$ in
$\Aut(G)$. We give a general criterion for the linearity of
$\Gamma_{G,A}$ in terms of surjections from $G$ to finite simple groups
of Lie type.
\end{abstract}

\maketitle
\tableofcontents
\section{Introduction}

In this paper, we investigate residual finiteness growth for certain
classes of groups, in relation to
linearity of their automorphism groups. Of particular interest to us are \emph{malabelian groups}, which are groups in which the centralizer of
every nontrivial conjugacy class is trivial. Typical examples of
malabelian groups are nonabelian free groups, hyperbolic surface groups, and
in general nonelementary hyperbolic groups. We are motivated particularly
by the question of the linearity of mapping class groups of surfaces
of finite type; this is an old question, which
is explicitly asked in Birman's
 1974 book~\cite{birman1974braids} (Problem 30 in the appendix).
 In general, this question is well-known and appears
 in both Farb's~\cite{Farb-survey} and Birman's~\cite{birman-survey}
 articles in the 2006 ``Problems in Mapping Class Groups" volume; see also~\cite{margalit-problems}.

We will develop the
machinery of residual finiteness growth of groups
that was originally introduced by Bou-Rabee~\cite{bou_rabee2010}, and
adapt it to the
study of automorphism groups of residually finite groups, thus
generalizing work of Bou-Rabee and McReynolds~\cite{bou_rabee-McReynolds2011,bou_rabee-McReynolds2016}.

\subsection{Residual finiteness growth}
Let $G$ be a finitely generated group, and fix a finite generating set $X$ for $G$. As is standard, for an element $g\in G$, we write $\|g\|_X$ for the minimal
length of a word representing $g$ in the generating set $X$.
\begin{defn}
We say that $G$ is \emph{residually finite} if for each nontrivial element $x \in G$, there exists an epimorphism $\varphi \colon G \longrightarrow Q$ to a finite group such that $\varphi(x) \neq 1.$ 
\end{defn}
The theory of \emph{effective residual finiteness}, also known as quantitative residual finiteness growth,
measures the difficulty of separating
a nontrivial element from the identity in a finite quotient.

To articulate these concepts precisely,
define the \emph{residual finiteness depth function} \[\D_G \colon G \backslash \{1\} \longrightarrow \mathbb{N} \cup \{\infty\}\]  by 
$$
\D_G(g) = \text{min}\{ |H| \: : \exists \: \varphi \colon G \longrightarrow H \text{ s.t. } |H| < \infty \text{ and } \varphi(g) \neq 1\},
$$
with the understanding that $\D_G(g) = \infty$ if no such finite quotient exists. By definition, $G$ is residually finite if and only if the function $\D_G(g)$ is finite for all nontrivial elements in $G$. Thus, we define
the \emph{residual finiteness growth function} $\RF_{G,X} \colon \mathbb{N} \longrightarrow \mathbb{N}$ by
$$
\RF_{G,X}(n) = \text{max}\{  \D_G(g) \: : \: \|g\|_X \leq n \text{ and } g \neq 1 \}.
$$
Given two finite generating sets $X_1$ and $X_2$, it is easy to see that $\RF_{G,X_1}(n) \approx \RF_{G,X_2}(n)$, i.e.~there are positive constants
$A_i$ and $B_i$ for $i\in\{1,2\}$ such that \[\RF_{G,X_1}(n)\leq A_1\cdot \RF_{G,X_2}(B_1\cdot n)\quad\textrm{and}\quad \RF_{G,X_2}(n)\leq A_2\cdot \RF_{G,X_1}(B_2\cdot n).\] Thus, when concerned with the coarse growth of the function $\RF$, we will suppress the
notation of the generating set and concern ourselves only
with the large scale behavior of the
function $\RF_G(n)$.

There is an extensive literature studying the asymptotic behavior for the function $\RF_G(n)$ and related functions for many classes of groups;
see \cite{dere2022survey} and the references
therein for an overview. A natural avenue for the study of
$\RF_{G}(n)$ is the
characterization of classes of groups $G$ based on the
large scale behavior of $\RF_{G}(n)$.

In the present work, we are most
interested in linearity of automorphism groups.
Finitely generated linear groups
are characterized
group theoretically by a result of Lubotzky~\cite{lubotzky-linear}, and
here we wish to give a criterion for linearity of the automorphism group of
a group $G$
in terms of the residual finiteness growth of $G$.
A result of Bou-Rabee--McReynolds~\cite{bou_rabee-McReynolds2011}
shows that for a finitely generated subgroup $G$ of a finite dimensional
linear group $\mathrm{GL}_{\ell}(\mathbb K)$, the growth of
$\RF_G(n)$ is bounded above by a polynomial function. A partial converse,
proved in later work of Bou-Rabee--McReynolds~\cite{bou_rabee-McReynolds2016},
says that hyperbolic groups $G$ for which there is a natural number $d$
and a constant $C>0$ such that $\RF_{G,S}(n) \leq C\cdot n^d$ can be
realized as subgroups of $\GL_{\ell}(\mathbb K)$; here
$\RF_{G,S}(n)$ is defined similarly to $\RF_G(n)$, except that the
homomorphisms are required to be epimorphisms to nonabelian finite simple
groups. The latter result applies more generally to \emph{uniformly
malabelian groups}, which we will define shortly and which are central to
the present work.

Following \cite{bou_rabee-McReynolds2016}, the above definitions are easily relativized to restricted classes of quotients:
\begin{defn}
    If $\mF$ is a class of finite groups, we
define $\D_{G,\mF}(x)$ identically to $\D_G(x)$, with the proviso that the target groups for the homomorphisms are required to lie in $\mF$. The residual finiteness growth function $\RF_{G,\mF}(n)$
 is defined by maximizing $\D_{G,\mF}(x)$ over the $n$-ball with respect to a finite generating set.
\end{defn}

Except for when we discuss finite simple groups of Lie type,
the symbol $G$ will refer to an infinite
group with trivial center. We will also assume, unless otherwise noted,
that $G$ is residually finite; this latter assumption implies that
$\Aut(G)$ is residually finite. Since $G$ is center-free, we have $G \cong \Inn(G)$. Each subgroup $A \leq \Out(G)$ gives rise to extension of $G$ written as
\[
1 \longrightarrow G \longrightarrow \Gamma_{G,A} \longrightarrow A \longrightarrow 1,
\]
where $\Gamma_{G,A} = q^{-1}(A)$, and where here $q \colon \Aut(G) \longrightarrow \Out(G)$ is the natural projection.
\begin{defn}
  If $A \leq \Out(G)$ is a subgroup, we define $\D_{G, \mF^A}(x)$ identically to $\D_{G, \mF}(x)$ except the quotients appearing in the depth function
are required to be $\Gamma_{G,A}$--invariant (i.e.~the kernel must be
invariant under the conjugation action of $\Gamma_{G,A}$). The function
$\RF_{G,\mF^A}(n)$ is defined analogously, by maximizing $\D_{G,\mF^A}(x)$ over the $n$-ball with respect to a finite generating subset.
\end{defn}

A group $G$ is said to be \emph{malabelian} if for every pair $g,h\in G$ of nontrivial elements, there is a conjugate $khk^{-1}$ of $h$ such that
$[g,khk^{-1}]\neq 1$; a finitely generated group
$G$ is said to be \emph{uniformly malabelian} if there is a constant $\kappa>0$ such that the element $k$ can be chosen to satisfy $\|k\|_X \leq \kappa$; in other words, a finitely generated group $G$ is uniformly malabelian if and only if
there exists a finite set $T\subseteq G$ such that for any nontrivial
$g,h\in G$, we have $[g,khk^{-1}]\neq 1$ for some $k\in T$.
Nonabelian free groups, surface groups, and in general all
nonelementary hyperbolic groups are examples of uniformly malabelian groups. 
Thompson's group F provides an example of a malabelian group that is not hyperbolic. We will discuss malabelian groups
in more detail in Section~\ref{ss:malabelian}.

Finite simple groups of Lie type will figure prominently in this paper;
the reader may find definitions and
a discussion in Section~\ref{ss:finite-simple}. A finite simple
group $H=H(q)$ of Lie type comes in one of finitely many families, and
the parameter $q=p^n$ parametrizes a finite extension of a prime field
$\mathbb F_p$. We say that a class $\mathcal H=
\{H_i(q_i)\}_{i\in\N}$ of finite simple
groups of Lie type are \emph{extension-bounded} if there is an $e\in \N$
such that for each $i$, the parameter $q_i$ satisfies $q_i= p_i^{n_i}$
with $n_i\leq e$. For a fixed $e$ which works for a class $\mathcal H$,
we say $\mathcal H$ is \emph{$e$--extension-bounded}.

\begin{thm}\label{thm:main-growth}
Let $G$ be a finitely generated, residually finite,
uniformly malabelian group. Suppose that:
\begin{itemize}
    \item $G$ has an infinite order element;
    \item $A\leq \Out(G)$ is a finitely generated subgroup;
    \item $\mF$ denotes the class of finite products of finite simple groups of Lie type;
    \item for each $e\in\N$, the class $\mF_e\subseteq\mF$ denotes a collection of finite products of $e$--extension-bounded
    finite simple groups of Lie type.
\end{itemize}
Then the following hold:
\begin{enumerate}
\item
Suppose that
there is a finite index subgroup $B \leq \Gamma_{G,A}$, a $B$-invariant finite index normal subgroup $H \trianglelefteq G$, and natural numbers $d$
and $e$ such that \[\RF_{H,\mF_e^{B/ H}}(n)\preceq n^d.\] Then
there exists a field $\mathbb K$ and a natural number $\ell$ such that $\Gamma_{G,A}\leq \GL_{\ell}(\bK)$.
\item
Suppose conversely that
$\Gamma_{G,A}\leq \GL_{\ell}(\bK)$. Then there exists a finite index subgroup $B \leq \Gamma_{G,A}$, a $B$-invariant finite index normal subgroup $H \trianglelefteq G$, and a natural number $d$ such that
\[\RF_{H,\mF^{B/ H}}(n)\preceq n^d.\] Moreover, if $\mathbb K$ has
characteristic zero  then for some $e\in\N$, we
have \[\RF_{H,\mF_e^{B/ H}}(n)\preceq n^d.\]
\end{enumerate}
\end{thm}

\subsection{Mapping class group motivation}
As we have alluded above, the principal motivation for this work is the linearity problem for the mapping class group.
Let $S_g$ be a closed oriented surface of genus $g\geq 2$, and let
$S_{g,1}$ be a genus $g$ surface with one boundary component.  The
Dehn--Nielsen--Baer theorem identifies the extended mapping class group
$\operatorname{Mod}^{\pm}(S_g)$ with $\operatorname{Out}(\pi_1(S_g))$,
with the orientation-preserving mapping class group corresponding to an
index-two subgroup.  For $S_{g,1}$, the corresponding version identifies
the mapping class group with the subgroup of $\operatorname{Aut}(\pi_1(S_{g,1}))$
preserving the boundary class, with the precise formulation depending on
whether mapping classes are required to fix the boundary pointwise.  Since
$\pi_1(S_g)$ is a closed surface group and $\pi_1(S_{g,1})\cong F_{2g}$,
these are uniformly malabelian examples.
The linearity problem is precisely formulated as whether or not $\operatorname{Mod}^{\pm}(S_g)$ admits a faithful, finite dimensional representation in characteristic zero, and is open for $g\geq 3$.

Thus, Theorem~\ref{thm:main-growth} may be viewed as a residual-finiteness-growth
reformulation of the linearity problem for these mapping class groups (and indeed a substantially larger class of automorphism groups, including handlebody groups and automorphism groups of free groups).  In
particular, after passing to the finite-index subgroups and invariant
finite-index normal subgroups appearing in the theorem, a polynomial bound
for the relevant invariant residual finiteness growth function, with respect
to finite products of extension-bounded finite simple groups of Lie type,
would imply linearity of the corresponding mapping class group.  Conversely,
a super-polynomial lower bound for every such function would obstruct
linearity by the contrapositive.

\subsection{Plan of the paper}
Sections~\ref{s:prelim-1} and~\ref{s:prelim-2} gather general facts about
finite simple groups and their automorphisms, ultraproducts of groups, 
malabelian groups, and finitely generated linear groups. Section~\ref{s:malabelian} gathers facts about semisimple
quotients of groups, especially with regard to malabelian groups. The
main general results relating residual finiteness and linearity are proved in Section~\ref{s:rf-lin}.

\section{General group theoretic preliminaries}\label{s:prelim-1}

\subsection{Generalities on groups}
The basic reference for this section is~\cite{aschbacher}.
We adopt the commutator convention $[x,y] = x^{-1}y^{-1}xy$. For a normal subgroup $H \trianglelefteq G$, we write $q_H \colon G \longrightarrow G/H$ for the natural projection, and $q = q_H$ and $\bar{x} = q_H(x)$ when the subgroup $H$ is clear from context. The letter $q$ will generally be reserved for quotients of groups or for a power of a prime; this will
generally not lead to confusion.

We will generally write $1=1_G$ for the identity element of a group $G$, and the trivial group will
be distinguished by $\{1\}$. As is standard, for a finite group $G$ we write $|G|$ for its order, and for an element $x \in G$, we write $|x|$ for the
order of $x$, and following classical finite group theory notation we write
$m_1(G) = \max_{x \in G}|x| $.  For a finite generating set $X$ for $G$, we denote the length of $g \in G$ with respect to $X$ by $\|g\|_X$, and we suppress the subscript when the finite generating set is clear from context.
We let $D^{i}(G)$ be the $i^{th}$ term of the derived series of $G$. We denote the center of $G$ by $Z(G)$. The set of epimorphisms from $G$ to $H$ is written $\Epi(G,H)$.

We will reserve $\mathbb{K}$ for a field, with algebraic closure given by $\overline{\mathbb{K}}$. We write $\text{char}(\mathbb{K})$ for the
characteristic of $\mathbb K$ and write $\mathbb{F}_q$ for the field of $q$ elements. The field $\mathbb{K}(T_1, \ldots, T_s)$ is the field of rational functions in the variables $T_1, \ldots, T_s$ with coefficients in $\mathbb{K}$. Given a ring $R$ and a finite collection of indeterminates $\{T_1, \ldots, T_s\}$, we write the polynomial ring with $s$ variables with coefficients in $R$ as $R[T_1, \ldots, T_s]$. Given a subring $R \leq \mathbb{K}$, we denote the field of fractions of $R$ by $\text{Frac}(R)$. Given a collection of nonzero primes $S$ in an integral domain $R$, the ring
$R[\frac{1}{S}]$ is the localization of $R$ at $S$; for us, the rings
under consideration will be polynomial rings in finitely many variables
over the integers or over a finite field, their fraction fields, and subrings
of the field of fractions arising from finite sets of nonzero elements
in the polynomial rings. We write $\lcm\{m_1, \ldots, m_s\}$ for the least common multiple of the natural numbers $m_1, \ldots, m_s.$

\subsection{Schur multipliers and Schur covers}
The Schur multiplier $M(G)$ of a group $G$ was originally defined by Schur~\cite{schur1,schur2,wiegold-schur}, and can be viewed
as an obstruction to lifting projective linear representations of finite
groups to linear
representations. Much of the following discussion can be found in  \cite{schur_multiplier} and \cite[6.9]{weibel}.

The Schur multiplier $M(G)$ is identified with the second homology group
$H_2(G, \mathbb{Z})$. When $G$ is itself finite, then $M(G)$ is a finite abelian
group whose exponent divides the order of $G$.

Let $G$ be a fixed perfect group. Given any two perfect central extensions of $G$, written
\[
E_1: 1 \longrightarrow A_1 \longrightarrow H_1 \longrightarrow G \longrightarrow 1
\]
and
\[
E_2: 1 \longrightarrow A_2 \longrightarrow H_2 \longrightarrow G \longrightarrow 1,
\]
we say that $E_1$ \emph{covers} $E_2$ if there exists a homomorphism $f \colon H_1 \longrightarrow H_2$ making the diagram of extensions commute.

A perfect central extension is \emph{universal} if it uniquely covers any perfect central extension of $G$. We note that if $E_1$ and $E_2$ are universal central extensions of $G$, then $E_1$ covers $E_2$ and $E_2$ covers $E_1.$ A group $G$ admits a universal central extension if and only if $G$ is perfect. When $G$ admits a universal central extension, then
this universal central extension is called the \emph{Schur cover} of $G$. The Schur cover of a perfect group $G$ is 
written $\tilde{G}$.

\subsection{Finite simple groups of Lie type}\label{ss:finite-simple}
We record some of the theory of simple linear algebraic groups and
groups of points fixed by Frobenius and Steinberg endomorphisms.
General references for this section are~\cite{Borel-linear,Humphreys-linear,malle_testerman}.

\subsubsection{Simple linear algebraic groups and finite groups of Lie type}
Let $\textbf{G}$ be a connected linear algebraic group defined over a field $\mathbb{K}$. We say $\textbf{G}$ is \emph{simple} if $\textbf{G}$ is non-abelian and does not admit any proper connected algebraic normal subgroups. We say that $\textbf{G}$ is \emph{semisimple} if every connected solvable algebraic normal subgroup is trivial.

We say that two $\mathbb{K}$-defined algebraic groups $\textbf{G}$ and $\textbf{H}$ are \emph{isogenous} if there exists a surjective $\mathbb{K}$-defined morphism from $\textbf{G}$ to $\textbf{H}$ with finite kernel; such a map is referred to as an
\emph{isogeny}.  A connected semisimple linear algebraic group
$\textbf{G}$ over a field $\mathbb{K}$ is \emph{simply connected} if every isogeny $f \colon \tilde{\textbf{G}} \longrightarrow \textbf{G}$ is an isomorphism. If $\textbf{G}$ is a $\mathbb{K}$-defined connected semisimple linear algebraic group, then there exists a natural isogeny
\[\begin{tikzcd}
	{\textbf{G}_{sc}} & {\textbf{G}}
	\arrow["\pi", from=1-1, to=1-2]
\end{tikzcd}\]
from a simply connected group $\textbf{G}_{sc}$; the kernel of $\pi$ lies in
the center of $\textbf{G}_{sc}$. The group $\textbf{G}_{sc}$ is unique within its isogeny class, which in turn is determined by a
Dynkin diagram and an indecomposable root system.

Up to isogeny, the classical simple linear algebraic groups over any algebraically closed field correspond to the Dynkin diagrams of the form
\[
A_n \: (n \geq 1), \quad B_n \: (n \geq 2), \quad C_n \: (n \geq 3), \quad D_n \: (n \geq 4)
\]
with the exceptional Dynkin diagrams given by
\[
E_6, \quad E_7, \quad E_8, \quad F_4, \quad G_2.
\]

Let $q$ be a power of the prime $p$. The map $F_q \colon \overline{\mathbb{F}}_q \longrightarrow \overline{\mathbb{F}}_q$ given by $t \longrightarrow t^q$ is called the \emph{$q$--Frobenius automorphism} and fixes the subfield $\mathbb{F}_q$
pointwise. Given a linear algebraic group $\textbf{G}$ defined over
$\overline{\mathbb{F}}_q$ equipped with an embedding $\textbf{G} \hookrightarrow \GL_{\ell}(\overline{\mathbb{F}}_q),$ the map $F_q \colon \textbf{G} \longrightarrow \textbf{G}$ given by \[(a_{ij}) \longrightarrow (a_{ij}^q),\] is a group homomorphism with fixed point subgroup
$$
\textbf{G}^{F_q} = \{ g \in \textbf{G} \: : \: F_q(g) = g\}.
$$
We write $\textbf{G}(q)$ for this subgroup.
We call $F_q$ the \emph{standard Frobenius} of $\textbf{G}$ with respect to $\mathbb{F}_q$. While this map is an isomorphism of groups, it is not an isomorphism of algebraic groups because it is generally not an isomorphism of varieties.

Let $\textbf{G}$ be a connected linear algebraic group defined over $\overline{\mathbb{F}}_p$. A surjective endomorphism $F \colon \textbf{G} \longrightarrow\textbf{G}$ of linear algebraic groups
which has only finitely many fixed points is called
a \emph{Steinberg endomorphism} of $\textbf{G}$. 
We write $\textbf{G}^F$ for the group of fixed points of $F$ on $\textbf{G}$. If $\textbf{G}$ is a semisimple algebraic group defined over $\overline{\mathbb{F}_q}$ with $q = p^f$ with a Steinberg endomorphism $F \colon \textbf{G} \longrightarrow \textbf{G}$, then the finite group of
fixed points 
\[
\textbf{G}^F = \{ g \in \textbf{G} \: : \: F(g) = g\}
\]
is called a \emph{finite group of Lie type}.

If $\{G_i(q_i)\}_{i\in\N}$ is a sequence of finite groups of Lie type,
where $q_i=p_i^{n_i}$, then we say that $\{G_i(q_i)\}_{i\in\N}$ is
\emph{extension-bounded} if there is an $e\in\N$ such that $n_i\leq e$
for all $i\in\N$. For such a class $\{G_i(q_i)\}_{i\in\N}$ and $e$, we say
that $\{G_i(q_i)\}_{i\in\N}$ is \emph{$e$--extension-bounded}.

A classical theorem of Tits
specifies which of
the finite groups of Lie type are simple, modulo their centers,
thus giving rise to \emph{finite simple groups of Lie type}.
\begin{thm}[Tits]\label{tits_thm}
Let $\textbf{G}$ be a connected, simply connected simple linear algebraic group defined over $\overline{\mathbb{F}}_p$ with a Steinberg endomorphism $F \colon \textbf{G}(\overline{\mathbb{F}}_p) \longrightarrow \textbf{G}(\overline{\mathbb{F}}_p).$ Then $\textbf{G}^F$ is perfect and that $\textbf{G}^F / Z(\textbf{G}^F)$ is simple, unless $\textbf{G}^F$ is one of
$$
\SL_2(2),\, \SL_2(3),\, \SU_3(2),\, \Sp_4(2),\, G_2(2),\, \: ^2B_2(2),\, \: ^2G_2(3),\, \:^2F_4(2).
$$
\end{thm}

The finite simple groups of Lie type, their
Schur multipliers and corresponding Schur covers, are all well-known; the
reader may find these listed in~\cite{malle_testerman}, tables
24.2 and 24.3. See also \cite[Remark 9.17]{malle_testerman} for more details.

One fact we will require is the following, which can be seen from
examining the orders of finite simple groups of Lie type:
\begin{lemma}\label{lem:q-divides}
    Suppose $G(q)$ is a finite simple group of Lie type, where here
    $q=p^n$. Then $q$ divides $|G(q)|$.
\end{lemma}

From examining the orders of general linear groups, we have the
following immediate corollary:

\begin{cor}\label{cor:lin-extension-bound}
    Suppose $q=p^n$ for some $n\in\N$ and let $G(q)$ be a quotient
    of a subgroup $Q\leq \GL_{\ell}(p)$.
    Then $n\leq \binom{\ell}{2}$.
\end{cor}

Let $G$ be a center-free finitely generated group, and let $A \leq \Out(G)$ be a finitely generated group. If $N\leq G$ is a normal subgroup such that $Q=G/N$ is isomorphic to a finite
direct product of (possibly different) finite simple groups of Lie type, then $Q$ is a quotient of \emph{semisimple type}, and if $N$ is $\Gamma_{G,A}$-invariant, we say that $Q$ is an $A$-invariant quotient of semisimple type. If $\mathcal Q$ is a family of semisimple type
groups, we say that this family is \emph{extension-bounded} if the family
$\mathcal H$ of finite simple groups of Lie type occurring as factors
of elements of $\mathcal Q$ is extension-bounded.

\subsubsection{Ultraproducts of nonabelian finite simple groups}
For a more detailed discussion of the following material, we refer the reader
to \cite{schoutens}; for general background on ultraproducts and ultrafilters, the reader may consult Section 1.6 in~\cite{hinman-book}. By a \emph{non-principal ultrafilter} $\omega$ on
an infinite set $X$, we mean a collection of subsets of $X$ 
which is:
\begin{enumerate}
    \item Closed under taking finite intersection.
    \item Closed under taking supersets.
    \item Does not contain a least element.
    \item Exhaustive, in the sense that for
    all $D \subset X$, either $D$ or its complement $D^c$ belongs to $\omega$.
\end{enumerate}

In particular, the empty set does not belong to $\omega$. Because $\omega$ is non-principal (i.e.~does not contain a least element), it follows that any co-finite subset belongs to $\omega$. The existence of non-principal ultrafilters follows from the Axiom of Choice, and for any infinite subset $A\subseteq X$
one can find a non-principal ultrafilter $\omega$ on $X$ containing $A$ as 
an element.

Let $\omega$ be a non-principal ultrafilter on $\mathbb{N}$, and let
$\{X_i\}_{i \in \mathbb{N}}$ be a family of nonempty sets. For
\[(x_i), (y_i) \in \prod_{i=1}^\infty X_i\] we write $(x_i) \sim_{\omega} (y_i)$
if and only if $\{ i \: : \: x_i = y_i \} \in \omega$. It is easy to see that $\sim_{\omega}$ forms an equivalence relation on $\prod_{i=1}^\infty X_i$.
Given $(x_i) \in \prod_{i=1}^\infty X_i$, we denote the equivalence class of $(x_i)$
by $(x_i)_\omega$. The \emph{ultraproduct} of the sets
$\{X_i\}_{i \in \mathbb{N}}$
along $\omega$ is given by
$$
X_{\omega} = \left(\prod_{i=1}^\infty X_i \right) \Bigg/ \sim_{\omega}.
$$
Choosing a nonempty subset $Y_i \subset X_i$ for each $i$,
we have $\prod_\omega Y_i$ is canonically identified with a subset of
$\prod_\omega X_i$.

Taking the ultraproduct of a collection of groups $\{G_i\}_{i\in\N}$,
their ultraproduct is naturally a group which is given by
\[G_{\omega}=\left(\prod_{i=1}^{\infty} G_i\right) \Bigg/ N_{\omega},\] where $N_{\omega} = \{ (1_{G_i})_\omega  \}$. An ultraproduct of rings is defined
similarly; it is a standard fact that 
an ultraproduct of fields is again a field which will be
algebraically closed if each factor is algebraically closed.
If $\{\mathbb{K}_i\}_{i\in\N}$ consists of finite fields where each prime characteristic appears at most finitely many times,
then the ultraproduct $\mathbb{K}_{\omega}$ has characteristic $0$. 

Returning to finite simple groups,
if $G$ is a finite simple group of Lie type, there exists a connected, simply connected simple linear algebraic group $\textbf{G}$ and a Steinberg endomorphism such that $\textbf{G}^F / Z(\textbf{G}^F) = G$. We will call $\textbf{G}$ the simple algebraic group \emph{associated} to $G$. Given a finite simple group of Lie type $G = \textbf{G}^T/Z(\textbf{G}^T)$ defined over the algebraic closure of $\mathbb{F}_q$ where $q = p^f$ for some prime $p$, we say that $p$ is the \emph{defining characteristic of $G$} or that $G$ is a finite simple group of Lie type in characteristic $p$. Additionally, we will write $p=\dc(G)$ and say that $G$ is a finite simple group of Lie type in characteristic $p$. When $G = \textbf{G}(q)/Z(\textbf{G}(q))$, we call \emph{$\mathbb{F}_q$ the defining field of $G$}.

We say an infinite collection $\{G_i\}_{i=1}^{\infty}$ of finite products of finite simple groups of Lie type has 
\emph{bounded multiplicity} if there exists a natural number 
$N \in \mathbb{N}$ such that each $G_i$ is isomorphic to a product of at most
$N$ finite simple groups of Lie type.

\subsubsection{Bounds on automorphism groups}
Let $G$ be a finite simple group of Lie type with associated connected, simply connected simple linear algebraic group $\textbf{G}$, defined over $\overline{\mathbb{F}}_p$,
and let \[F \colon \textbf{G}(\overline{\mathbb{F}}_p) \longrightarrow \textbf{G}(\overline{\mathbb{F}}_p)\] be a Steinberg endomorphism
such that $G = \textbf{G}^F / Z(\textbf{G}^F)$.
The next lemma constructs a faithful representation
\[\rho \colon \Aut(G) \longrightarrow \GL_\ell(\overline{\mathbb{F}}_p),\] wherein
$\ell$ depends only on the degree of a faithful projective representation of $G$ and the degree of defining field over the prime field.

\begin{lemma}\label{lemma: rep dimension of Aut}
    Let $G$ be a finite simple group of Lie type, and let $\mathbb{F}_{p^\ell}$ be the defining field of $G$. There exists a constant $C>0$ such that if $d$ is an integer with $G \leq \PGL_d(\mathbb{F}_{p^\ell})$, then \[\Aut(G) \leq \GL_{C\ell d^3} (\mathbb{F}_{p}).\] 
\end{lemma}
\begin{proof}
From
     \cite[Theorem 30 and 36]{steinberg}, we have that every automorphism of $G$ is the
     composition of an inner automorphism, a diagonal automorphism, a graph automorphism (i.e.~induced by an automorphism of the Dynkin diagram), and a field automorphism. Since $G \cong \Inn(G)$, we have that $\Out(G)$ is generated by diagonal, graph, and field automorphisms.
     Here, diagonal automorphisms are the automorphisms induced by conjugation inside a suitable adjoint algebraic group; modulo inner automorphisms, they form the diagonal automorphism group.
     From \cite[Exercise pg. 96]{steinberg}, we have that if $D$ is the group of diagonal automorphisms modulo those that are inner, then $D$ is isomorphic to the center of the Schur cover of $G$. Examining
     tables 24.2 and 24.3 in~\cite{malle_testerman} and comparing them
     to the bounds on the values found in Theorem A.2~\cite{bradford_thom_shortlaws} or Proposition 5.4.13 of~\cite{KL}, there exists a constant $C>0$ such that $|D| \leq C \cdot d.$
     
     The automorphisms of $G$ induced by field automorphisms form a cyclic group generated by the Frobenius map
     \[F_p \colon \textbf{G}(p^\ell)
     \longrightarrow \textbf{G}(p^\ell),\] where $\ell$ is the order of the standard Frobenius automorphism $F_p$ on $G$. Graph automorphisms are induced by automorphisms of the Dynkin diagram, and their outer classes have order either $2$ or $3$.
     
     Let $C_\ell$ be the cyclic group of order $\ell$ with generator $y$. If $C_\ell$ acts on $G$ via $x \cdot a = F_p(a)$, then the previous
     remarks show that
     $G \rtimes C_\ell$ has index at most  $3d$ in $\Aut(G)$, where here $G$ is identified with its group of inner
     automorphisms. Thus, if $m$ is a bound for the minimal dimension
     of a representation of $G \rtimes C_\ell$ over a given field $\mathbb K$, then from the induced representation,
     we obtain \[\Aut(G) \leq \GL_{3Cdm}(\mathbb K).\]
     
     Therefore, we may restrict our attention to representations
     of the group
     $G \rtimes C_\ell$. We may view $G \leq \GL_{w(G)}(p^\ell)$, where  $w(G)=d^2$ is the square of the values found in Theorem A.2~\cite{bradford_thom_shortlaws} or Proposition 5.4.13 of~\cite{KL}. The Frobenius map is not linear over $\mathbb F_{p^{\ell}}$, but
     $\mathbb F_{p^{\ell}}$ is an $\ell$--dimensional vector space over
     $\mathbb F_p$. After identifying $\mathbb F_{p^{\ell}}^{w(G)}$ with
     $\mathbb F_p^{\ell\cdot w(G)}$, the Frobenius map induces an
     $\mathbb F_p$--linear transformation. We let $C_\ell$ act by this
     transformation, applying Frobenius to each coordinate in
     $\mathbb F_{p^{\ell}}^{w(G)}$.

We claim that $G \rtimes C_\ell$ 
admits a faithful representation over $\mathbb{F}_{p}$ via
\[(g,x^t)\cdot v = g\cdot x^t(v),\] where $0\leq t<\ell$.
It is easy to see that each of the above maps is linear. We need to show that we have obtained a homomorphism. Note that
\begin{eqnarray*}
    (g_1,x^{t_1}) \cdot ((g_2, x^{t_2}) \cdot v) &=& (g_1, x^{t_1})(g_2 \circ x^{t_2})(v)\\
    &=&g_1 \circ x^{t_1} \circ g_2 \circ x^{t_2}(v)\\
    &=&g_1 \circ x^{t_1} \circ g_2 \circ x^{-t_1} \circ x^{t_1 + t_2}(v)\\
    &=&g_1 \circ F_p^{t_1}(g_2) \circ x^{t_1 + t_2}(v)\\
    &=&(g_1 F_p^{t_1}(g_2), x^{t_1 + t_2})(v)\\
    &=&((g_1, x^{t_1}) \cdot (g_2, x^{t_2}))(v).
\end{eqnarray*}

We thus have an action of $G \rtimes C_\ell$ on $\mathbb{F}_{p}^{\ell\cdot w(G)}$. If this action were not faithful, then, since the restriction to $G$ is faithful, there would be some element $(g,x^t)$ in the kernel with $x^t\neq 1$. By conjugating suitably, we see that $(g',x^t)$ also lies in the
kernel for some $g'\neq g$, whence $(g^{-1}g',\mathrm{id})$ lies in the kernel.
Since the restriction of the action to $G$ is faithful, this is a contradiction.
We have thus found a faithful representation \[\varphi \colon G \rtimes C_\ell \longrightarrow \GL_{\ell\cdot w(G)}(p),\] as desired.
\end{proof}

Let $G$ be a finite simple group of Lie type with defining field
$\mathbb{F}_{p^\ell}$, and let $m\in\N$. We have the following corollary, which bounds the minimal dimension of a representation over $\mathbb{F}_{p}$ of $\Aut(G^m)$ from above in terms of the minimal 
dimensional $\mathbb{F}_{p^\ell}$--representation of $G$ and the integer $m$. 
\begin{cor}\label{cor: rep dimension of Aut of product}
    Let $G$ be a finite simple group of Lie type with defining field $\mathbb{F}_{p^\ell}$, and let $d$ be the minimal degree of a projective representation of $G$ over
    $\mathbb{F}_{p^\ell}$. There exists a universal constant $C>0$ such that $\Aut(G^m) \leq \GL_{C(m !) m\ell d^3}(p)$ for all $m \in \mathbb{N}$.
    \end{cor}
    \begin{proof}
        Since $G$ is a finite simple group, we have that  \[\Aut(G^m) = \Aut(G)^m \rtimes \Sym(m),\] where the symmetric group $\Sym(m)$ acts on $\Aut(G)^m$ by permutation of coordinates. Indeed, every
        automorphism of $G^m$ must preserve the direct factors of $G^m$:
        suppose $g\in G^m$ is given by $(x,1,\ldots,1)$, where only the
        first coordinate is nontrivial, and this element is sent
        by an automorphism to an element $h$
        which has at least two nontrivial coordinates. Observe
        that the conjugacy class of $x$ in $G^m$ only generates one copy
        of $G$, whereas the conjugacy class of $h$ will generate a copy
        of $G$ in at least
        two coordinates.
        
        Lemma \ref{lemma: rep dimension of Aut} implies that \[\Aut(G) \leq \GL_{C\ell d^3}(p)\] for a universal constant $C>0$. Therefore, \[(\Aut(G))^m \leq \GL_{Cm\ell d^3}(p).\] Since $|\Sym(m)| = m!$, we have an induced representation $$
        (\Aut(G))^m \rtimes \Sym(m) \leq  \GL_{C (m !) m\ell d^3}(p)
        $$ as desired.
    \end{proof}

For each prime $p \in \mathbb{N}$, we let  $r_p(G)$ be the minimal positive integer $\ell$ for which there is a natural number $t \in \mathbb{N}$ and an injective homomorphism
\[\varphi \colon G \longrightarrow \PGL_{\ell}(p^t).\] We define
\[r(G) = \min_{p \text{ prime }}r_p(G),\] and define $r_p^L(G)$ and $r^L(G)$
in the same fashion, substituting $\GL_{\ell}$ for the role of $\PGL_{\ell}$.
When $G$ is simple, we clearly have $r(G) \leq r^L(G)$. Additionally, since
\[\PGL_{\ell}(K) \leq \GL_{\ell^2}(K)\] for an arbitrary field $K$, we have
$r^L(G) \leq (r(G))^2$ for any group. We say a non-empty collection of finite groups $\mathcal{F}$ has \emph{bounded rank} if there exists a constant $R>0$ such that $r^L(G) \leq R$ for all $G \in \mathcal{F}$, and has \emph{bounded projective rank} if $r(G) \leq R$ for all $G \in \mathcal{F}$. 

By comparing the minimal dimensional faithful representation
of a finite simple group of Lie type over its defining field with
Theorem 5.3.9 in~\cite{KL}, we see:
\begin{prop}\label{prop:bounded-bounded}
    Let $\{G_i\}_{i\in\N}$ be a family of finite simple groups of Lie
    type, with $p_i$ the characteristic of the defining field of $G_i$.
    Then the set of natural numbers $\{r(G_i)\}_{i\in\N}$ is bounded
    if and only if the set $\{r_{p_i}(G_i)\}_{i\in\N}$ is bounded.
\end{prop}

In particular, Proposition~\ref{prop:bounded-bounded} allows one to
assume, up to a bounded error, that
minimal dimensional faithful representations of finite simple groups
of Lie type occur over the defining field.

The following lemma is inspired by \cite[Lemma 2.2]{bou_rabee-McReynolds2016}; here and throughout this paper, logarithms will be assumed
to be base two unless otherwise noted. Recall also that $m_1(H)=\max_{h\in H}|h|$ denotes the maximal order of an element of a finite group $H$.
\begin{lemma}\label{lem:rank_inequality1}
    Let $\{H_i^{\ell_i}\}_{i\in\N}$ be a set of finite products of $e$--extension-bounded nonabelian finite simple groups of Lie type. Then $\{r(\Aut(H_i^{\ell_i}))\}_{i\in\N}$ is bounded if and only if the sequences 
    $\left\{ \ell_i\right\}_{i \in \mathbb{N}}$ and
    \[\left\{\frac{\log |H_i^{\ell_i}|}{\log(m_1(H_i^{\ell_i}))}\right\}_{i\in\N}\] are both bounded.
\end{lemma}
\begin{proof}
    Suppose the sequence $\{r(\Aut(H_i^{\ell_i}))\}_{i\in\N}$ is bounded.
    We then have that the sequence $\{r(H_i^{\ell_i})\}_{i\in\N}$ is also bounded, since
    \[
    H_i^{\ell_i} \leq \Aut(H_i^{\ell_i}).
    \]
    
    Since the sequence $\{r(H_i^{\ell_i})\}_{i\in\N}$ is bounded, we have that $\{\ell_i\}_{i\in\N}$ is bounded by some integer $\ell$. To see this, suppose otherwise for a contradiction. We then have the collection $\{H_i^{\ell_i}\}_{i\in\N}$ contains subgroups of the form $C^{\ell_i}$, where
    $C$ is a fixed nontrivial cyclic group and $\ell_i$ can achieve arbitrarily large
    values. We may assume that $C$ is not divisible by $p$ since the
    ambient groups are not nilpotent (or, by appealing to Feit--Thompson's Odd Order
    Theorem). Passing to the
    algebraic closure of the defining field, we see the action of $C^{\ell}$ is diagonalizable.
    Since the multiplicative group of a finite field is cyclic, it follows that
    $r(H_i^{\ell_i})\geq\ell_i$ for all $i$, which is a contradiction. Additionally, it follows that the sequence $\{r(H_i)\}_{i\in\N}$ is bounded, since $H_i \leq \Aut(H_i^{\ell_i})$. Because \[\log(m_1(H_i)) \leq \log(m_1(H_i^{\ell_i})),\]
    it follows from \cite[Lemma 2.2]{bou_rabee-McReynolds2016} that
    $$
\frac{\log |H_i^{\ell_i}|}{\log(m_1(H_i^{\ell_i}))} \leq \frac{\log |H_i^{\ell_i}|}{\log(m_1(H_i))} \leq \ell \frac{\log |H_i|}{\log(m_1(H_i))} \leq K
    $$
    for some constant $K > 0$.
    
    Conversely, suppose that both of
    the sequences $\left\{ \ell_i\right\}_{i\in\N}$ and
    \[\left\{\frac{\log |H_i^{\ell_i}|}{\log(m_1(H_i^{\ell_i}))}\right\}_{i\in\N}\] are bounded by $R>0$. Since $\ell_i\leq R$, we have
    \[
    \log |H_i^{\ell_i}| = \ell_i\log |H_i| \leq R\log |H_i|.
    \]
    For all elements in $H_i^{\ell_i}$, the following inequality holds:
    \[
    |(x_1,\ldots, x_{\ell_i})| =\lcm\{|x_1|,\ldots, |x_{\ell_i}|\} \leq \prod_{t=1}^{\ell_i}|x_t| \leq (m_1(H_i))^{\ell_i}.
    \]
    Therefore,
    \[
    m_1(H_i^{\ell_i}) \leq (m_1(H_i))^{\ell_i} \leq (m_1(H_i))^R.
    \]
    Combining this with the assumed bound
    \[
    \frac{\ell_i\log |H_i|}{\log(m_1(H_i^{\ell_i}))}\leq R
    \]
    gives
    \[
    \frac{\log |H_i|}{\log(m_1(H_i))} \leq R^2.
    \]
    From \cite[Lemma 2.2]{bou_rabee-McReynolds2016}, we see that $\{r(H_i)\}_{i \in \mathbb{N}}$ is bounded. Since $\ell_i \leq R$ for all $i$ and the family $\{H_i\}_{i\in\N}$ is extension-bounded,
    Corollary \ref{cor: rep dimension of Aut of product} implies $\{r(\Aut(H_i^{\ell_i}))\}_{i \in \mathbb{N}}$ is bounded.
\end{proof}

The following is well known; see \cite{Hall95} for instance.
\begin{lemma}\label{lem:ultraproduct representation}
    If $\mathcal{F} = \{G_i\}_{i \in \mathbb{N}}$ is a set of finite groups 
    such that either the rank or the projective rank of
    elements in $\mathcal F$
    is bounded by some $R \in \mathbb{N}$, then for any non-principal ultrafilter $\omega$ on
    $\N$ there is an injective homomorphism \[\varphi_\omega \colon G_\omega \longrightarrow \GL_\ell(\mathbb{K})\] for some
    $\ell \in \mathbb{N}$ and some field $\mathbb{K}$.
\end{lemma}

\section{Preliminaries on geometric group theory and linear groups}\label{s:prelim-2}

\subsection{Malabelian groups}\label{ss:malabelian}
Recall that a group $G$ is \emph{malabelian} if for any pair of nontrivial elements $g,h \in G$ (possibly with $g=h$), there exists an element $k \in G$ such that $[g,khk^{-1}] \neq 1$. In other words, a group $G$ is malabelian if every nontrivial conjugacy class in $G$ has a trivial centralizer. 

Recall that a finitely generated group $G$ is $\kappa$-\emph{malabelian} with respect to a finite generating set $X$ if for every pair of nontrivial elements $a,b \in G$, there exists an element $k \in G$ with $\|k\|_X \leq \kappa$ such that $[kak^{-1},b] \neq 1.$ If $G$ is $\kappa$-malabelian with respect to a finite generating set $X$ and $X'$ is some other finite generating set, then $G$ is $\kappa'$-malabelian with respect to $X'$ for some other $\kappa' \in \mathbb{N}$, since the
corresponding word metrics on $G$ are bi-Lipschitz to each other.
We may say that $G$ is
\emph{uniformly malabelian} if the constant $\kappa$ is not specified,
and that any $\kappa$ as above is a \emph{uniformly malabelian
constant} with respect to $X$. Since centralizers of nontrivial elements
in free groups and
closed surface groups are cyclic, we easily obtain:
\begin{prop}\label{prop:free-surface-malabelian}
    Finitely generated nonabelian free groups and surface groups are
    uniformly malabelian.
\end{prop}

More generally, nonelementary hyperbolic groups are uniformly malabelian,
though we will not require this fact.
Let $G$ be a finitely generated uniformly malabelian group, and let $\ell \in \mathbb{N}$. The following proposition gives an upper bound on the minimal length of a nontrivial element of the $\ell^{th}$ term of the derived series of $G$ in terms of $\ell$. The following lemma will be useful for bounding $\RF_{G, \mathcal{F}^A}(n)$, for various families $\mathcal F$
of products of finite simple groups of Lie type.
\begin{lemma}\label{lem:bound length of element in D^n(G)}
    Suppose that $G$ is a finitely generated uniformly malabelian group with a finite generating set $X$. Let $\kappa$ be a uniformly malabelian constant of $G$ with respect to $X$, and let
    $1\neq a\in G$ be arbitrary. Then for all $n \in \mathbb{N}$, there exists a word $w_{n,a} \in D^n(G)$ such that the following hold:
    \begin{enumerate}
    \item $\|w_{n,a}\|_X \leq 8^n\max\{\|a\|_X, \kappa\}$;
    \item If $\varphi \colon G \longrightarrow Q$ is an epimorphism such
    that $\varphi(w_{n,a}) \neq 1$, then $\varphi(a) \neq 1$;
    \item If $\varphi \colon G \longrightarrow Q$ is an epimorphism and $N$ is a normal subgroup of $Q$ such that $\varphi(a) \in N,$ then $\varphi(w_{n,a}) \in D^n(N).$
    \end{enumerate}
\end{lemma}
\begin{proof}
    We proceed by induction on $n$. For the base case, there exists an element $k \in G$ with $\|k\|_X \leq \kappa$, such that $w_{1,a} = [a, kak^{-1}] \neq 1$. We see that
    $$
    \|w_{1,a}\|_X  \leq 2\|a\|_X + 2\|kak^{-1}\|_X \leq 4\|a\|_X + 4\|k\|_X \leq 8\max\{\|a\|_X,\kappa\}.
    $$
    Moreover, if $\varphi \colon G \longrightarrow Q$ is an
    epimorphism such that $\varphi(a) = 1$, then clearly $\varphi([a,kak^{-1}]) = 1$, as desired. Note that if $\varphi(a) \in N$ and $N$ is a normal subgroup of $Q$, then $\varphi(kak^{-1}) \in N$ as
    well, whence, $\varphi([a,kak^{-1}]) \in D^1(N).$

    For $n \geq 2$, by induction one obtains a nontrivial element $w_{n - 1,a} \in D^{n - 1}(G)$ such that 
    $$
    \|w_{n - 1,a}\|_X \leq 8^{n-1}\max\{\|a\|_X, \kappa\},
    $$ such that if $\varphi \colon G \longrightarrow Q$ is an epimorphism with $\varphi(w_{n-1,a}) \neq 1$ then $\varphi(a) \neq 1$, and such that
    if $\varphi\colon G\longrightarrow Q$ is an epimorphism 
    and $N$ is a normal subgroup of $Q$ where $\varphi(a)\in N$, then $\varphi(w_{n-1,a})\in D^{n-1}(N)$.
    
    Since $G$ is uniformly malabelian, there exists an element $k \in G$ with $\|k\|_X \leq \kappa$ such that 
    $$
    w_{n,a} = [w_{n-1,a},kw_{n-1,a}k^{-1}] \neq 1.
    $$ 
    Since $w_{n - 1,a} \in D^{n - 1}(G)$ and $D^{n-1}(G)$ is normal in $G$, we have $kw_{n - 1,a}k^{-1} \in D^{n - 1}(G).$ Therefore,
    $$
    w_{n,a} = [w_{n-1,a}, kw_{n - 1,a}k^{-1}] \in D^n(G).
    $$
    We observe that
    \begin{eqnarray*}
    \|w_{n,a}\|_X &\leq& 2\|w_{n -1,a}\|_X + 2\|kw_{n -1, a}k^{-1}\|_X\\
    &\leq& 4\|w_{n -1, a}\|_X + 4\kappa\\
    &\leq& 8\max\{\|w_{n-1,a}\|, \kappa\}\\
    &\leq& 8^{n}\max\{\|a\|_X,\kappa\}.
    \end{eqnarray*}
    
    Additionally, if $\varphi \colon G \longrightarrow Q$ is an epimorphism such that $\varphi(a) = 1$, we have
    $$
    \varphi(w_{n,a}) = \varphi([w_{n-1,a},kw_{n-1,a}k^{-1}]) = [\varphi(w_{n-1,a}), \varphi(kw_{n-1,a}k^{-1})] = 1.
    $$
    
    From the inductive hypothesis, if $\varphi(a) \in N$ for some normal subgroup of $Q$, then \[\varphi(w_{n-1,a}) \in D^{n-1}(N).\] Hence, $\varphi(kw_{n-1,a}k^{-1}) \in D^{n-1}(N)$ since $D^{n-1}(N)$ is normal in $N$. Therefore, 
    $$
    \varphi(w_{n,a}) = \varphi([w_{n-1,a}, kw_{n-1,a}k^{-1}]) = [\varphi(w_{n-1,a}), \varphi(kw_{n-1,a}k^{-1})] \in D^n(N),
    $$
    completing the proof of the lemma.
\end{proof}

Recall that if $G$ is a malabelian group and $A\leq \Out(G)$ is a
subgroup, then $\Gamma_{G,A}$ denotes the preimage of $A$ in
$\Aut(G)$.
For $N\leq G$ a subgroup and $A \leq \Out(G)$, we write $\mO_{N,A}$ for the orbit of $N$ under the conjugation action of $\Gamma_{G,A}$.
The \emph{$A$-invariant   } of $N$ is the intersection 
\[N_{A}=\bigcap_{M\in \mO_{N,A}} M.\] By construction, $N_{A}$ is a normal $\Gamma_{G,A}$-invariant subgroup in $G$. When $A = \Out(G)$, we will write $N_{\text{char}}$ and call $N$ the \emph{characteristic core} of $N$ in $G$.

\subsection{Linear groups}
In this section, we will gather some facts about finitely generated
groups of matrices, which will be useful in the sequel.
\begin{lemma}\label{lem:polynomial-localization}
    Let $G\leq\GL_\ell(\mathbb K)$ be a finitely generated subgroup. Then there exist:
    \begin{enumerate}
        \item A ring $\mathbb L\in\{\Z,\mathbb F_p\}$;
        \item A finite set of indeterminates $\{T_1,\ldots,T_s\}$;
        \item A finite set of nonzero polynomials
        $S\subseteq \mathbb L[T_1,\ldots,T_s]$;
        \item A faithful homomorphism \[G\longrightarrow
        \GL_{\ell'}\left(\mathbb L\left[\frac{1}{S}\right][T_1,\ldots,
        T_s]\right)\] for some $\ell'\in\N$.
    \end{enumerate}
\end{lemma}
\begin{proof}
    Since $G$ is finitely generated, we have that the image of $G$
    in $\GL_\ell(\mathbb K)$ is generated by a finite set of
    matrices, which
    we may assume is closed under taking inverses. Taking
    the subfield $\mathbb K_0\subseteq\mathbb K$ generated by these matrix entries, we see that $\mathbb K_0$ is a finite extension of
    $\Q(T_1,\ldots,T_s)$ or of $\mathbb F_p(T_1,\ldots,T_s)$, depending
    on the characteristic of $\mathbb K$ and on the transcendence degree of
    $\mathbb K_0$. Viewing $\mathbb K_0$ as a finite dimensional
    vector space over
    one of these rational function fields, say of degree $m$,
    we conclude that $G$ embeds in $\GL_{m\ell}$ over one of these function
    fields. By considering the denominators of the matrix entries of
    generators of $G$ in this representation, we see that the image of $G$ lies
    in the localization of $\mathbb L[T_1,\ldots,T_s]$ at a finite set of nonzero polynomials
    $S\subseteq \mathbb L[T_1,\ldots,T_s]$, as desired.
\end{proof}

The following is a standard fact due to Zassenhaus; the bound could
be sharpened but we will not require anything stronger:
\begin{prop}\label{prop:lin-solvable}
    There exists a universal constant $C$ such that if
    $\mathbb K$ is an arbitrary field and $S\leq \GL_{\ell}(\mathbb K)$ is a solvable subgroup, then the derived length of $S$ is at most
    $\lceil C\log(\ell)\rceil$.
\end{prop}

The following result of Larsen and Pink
appears as Theorem 0.2 in~\cite{larsen_pink}, and
is absolutely crucial for our present work:
\begin{thm}\label{thm:larsen-pink}
    Let $\mathbb K$ be a field and let $Q\leq\GL_\ell(\mathbb K)$ be a finite subgroup.
    Then there exists a constant $J(\ell)$ depending only on $\ell$ and normal
    subgroups \[Q_3\leq Q_2\leq Q_1\] of $Q$ such that the following
    conclusions hold:
    \begin{enumerate}
        \item $[Q:Q_1]\leq J(\ell)$;
        \item Either $Q_1=Q_2$, or $\mathbb K$ has positive characteristic $p$
        and $Q_1/Q_2$ is a direct product of finite simple groups
        of Lie type in characteristic $p$;
        \item The group $Q_2/Q_3$ is abelian of order not divisible
        by the characteristic of $\mathbb K$;
        \item The group $Q_3$ is either trivial, or $\mathbb K$ has positive characteristic $p$ and $Q_3$ is a $p$--group.
    \end{enumerate}
\end{thm}

For a fixed finite subgroup $Q\leq\GL_\ell( \mathbb K)$, we will call such subgroups
$(Q_1,Q_2,Q_3)$ a \emph{Larsen--Pink} triple for $Q$. Evidently, the
automorphism group of $Q$ acts on Larsen--Pink triples for $Q$.

\subsection{Matrix entries in linear groups}
Given a group $G \leq \GL_\ell(\mathbb{K})$ in characteristic $0$, it may be the case that $G$ is only definable over a transcendental extension of finite degree over $\mathbb{Q}$. Thus, we need to address polynomial rings in finitely many variables with coefficients in $\mathbb{Z}[\frac{1}{S}]$ with finitely many nonzero inverted polynomials. A similar situation arises in characteristic $p$. The following lemma allows us to reduce many of our considerations to the single variable case, in both zero and positive
characteristic. The following lemma and its proof can be originally found in \cite[Lemma 2.1]{BM15}, and we include details for the convenience of the
reader.
\begin{lemma}\label{lem:reduction to one variable}
    Let $f \in R[T_1,\ldots,T_s]$ be a nonzero polynomial of degree $d$ where $R = \mathbb{F}_p$ or $R = \mathbb{Z}$. Then there exists a sequence $\{n_i\}_{i=1}^s$ taking values in $\{0, 1, \ldots, d^{2s}\}$ such that if $\tau$ is an indeterminate, then
    $$
    0\neq f(\tau^{n_1}, \ldots, \tau^{n_s})\in R[\tau].
    $$
\end{lemma}
\begin{proof}
    We prove this by double induction on $s$ and $d = \deg(f)$, and we observe that the base cases of $s=1$ or $d=0$ are trivial. For the inductive case, let $f$ be a degree $d$ polynomial in $R[T_1, \ldots, T_s]$. We may write
    $$
    f(T_1, \ldots, T_s) = (h_0 + T_1h_1)T_1^k,
    $$
    where $h_0 \in R[T_2, \ldots, T_s]$ is nonzero, $h_1 \in R[T_1, \ldots, T_s]$, and $k\leq d$ a natural number. If $k > 0$, then the inductive hypothesis applied to $h_0 + T_1h_1$ (which has degree $<d$) gives the result. Otherwise, we may assume $k=0$. Since
    $h_0$ is a nonzero element of $R[T_2, \ldots, T_s]$, the inductive hypothesis implies there exist natural numbers $n_2, \ldots, n_s \in \{0, 1, \ldots, d^{2s-2}\}$ such that
    $$
    h_0(\tau^{n_2}, \ldots, \tau^{n_s}) \neq 0.
    $$
    If $h_1(\tau^{d^{2s}}, \tau^{n_2}, \ldots, \tau^{n_s}) = 0$, we have
    $$
    f(\tau^{d^{2s}}, \tau^{n_2}, \ldots, \tau^{n_s}) = (h_0(\tau^{n_2}, \ldots, \tau^{n_s}) + \tau^{d^{2s}} h_1(\tau^{d^{2s}}, \ldots, \tau^{n_s}))\tau^{kd^{2s}} = h_0(\tau^{n_2}, \ldots, \tau^{n_s}) \neq 0.
    $$
    Hence, we may assume $h_1(\tau^{d^{2s}}, \tau^{n_2}, \ldots, \tau^{n_s}) \neq 0.$ We then observe
    $$
    \deg(h_0(\tau^{n_2}, \ldots, \tau^{n_s})) \leq d\cdot d^{2s-2}= d^{2s - 1} < d^{2s} \leq \deg(\tau^{d^{2s}}h_1(\tau^{n_2}, \ldots, \tau^{n_s})).
    $$
    Thus,
    $$
    h_0(\tau^{n_2}, \ldots, \tau^{n_s}) \neq - \tau^{d^{2s}}h_1(\tau^{n_2}, \ldots, \tau^{n_s}).
    $$
    We conclude that
    $$
    f(\tau^{d^{2s}}, \tau^{n_2}, \ldots, \tau^{n_s})  \neq 0,
    $$
    as desired.
\end{proof}

Given $f \in R[T_1, \ldots, T_s]$ where $R = \Z$ or $R = \mathbb{F}_p$, we call a nonvanishing polynomial
$h \in R[\tau]$ as constructed by substitutions as in Lemma \ref{lem:reduction to one variable} a \emph{trace polynomial} for $f$. The next lemma gives a
controlled prime
number $p$ such that $f(m) \neq 0 \pmod p$ for some $0 \leq  m \leq \deg(h) +1$ when $R = \Z$.
\begin{lemma}\label{lem: integer polynomial nonvanishing}
    Let $f \in \mathbb{Z}[T_1, \ldots, T_s]$ be a nonzero polynomial, with $\deg(f) \leq d.$ Let
    \[h=a_0 + a_1\tau + \cdots + a_r\tau^r \in \mathbb{Z}[\tau]\] be a minimal degree
    trace polynomial for $f$.
    Then there exists a constant $C = C(s)$, a prime $p$, and a
    natural number $0 \leq m \leq d^{2s+1}+1$ such that
    $$
    p \leq C(\log(\max\{ |a_0|, \ldots, |a_r|\}) + (2s+2)d^{2s +2})
    $$
    and such that
    $$
    h(m) \neq 0 \text{ mod } p.
    $$
\end{lemma}
\begin{proof}
    Observe that if $f$ has a nonzero constant term then we may simply
    take $h=a_0$. The prime number theorem implies that there exists
    a universal constant $C_1$ and
    a prime $p$ not dividing $a_0$ of size $p\leq C_1\log |a_0|$; we may thus
    assume that $f$ has no constant term, whence $a_0=0$.

    By the construction of a trace polynomial $h$ in Lemma \ref{lem:reduction to one variable}, we have $r = \deg(h) \leq d^{2s+1}.$ Since $h$ has at most $r$ roots, there exists an integer $1 \leq m \leq r+1$ such that $h(m) \neq 0$ (since zero is automatically a root of $h$). Setting 
    $$
    A = \max\{|a_1|, \ldots, |a_r|\},
    $$
    it is easy to see that
    $$
    |h(m)| \leq r \cdot A\cdot m^r +A\leq r(m^r A)+m^r A=(r+1)(m^r A).
    $$
    The prime number theorem again implies there exists a prime $p$ such that $p \nmid |h(m)|$ and $p \leq C_1 \log(|h(m)|)$. It follows that
    \begin{eqnarray*}
    p &\leq& C_1 \log(|h(m)|)\\
    &\leq& C_1(\log(A) +  r \log(m) + \log(r+1))\\
    &\leq& C_1(\log(A) + d^{2s+1} \log(d^{2s+1} + 1) + \log(d^{2s+1}+1))\\
    &\leq& C_1(\log(A) + (d^{2s+1} + 1) \log(2d^{2s+1}))\\
    &\leq& C_1(2\log(A) + 2d^{2s+1}\log(2d^{2s+1}))\\
    &\leq& 2C_1(\log(A) + d^{2s+1}+(2s+1) d^{2s+2})\\
    &\leq& 2C_1(\log(A) + (2s+2) d^{2s+2}).
    \end{eqnarray*}
    We thus obtain the desired upper bound for the prime $p$ and for the integer $m$. Finally, we see that
    $$
    h(m) \neq 0 \pmod p,
    $$
    completing the proof.
\end{proof}

The following is the analogue of Lemma \ref{lem: integer polynomial nonvanishing} for characteristic $p$, and can be found as Lemma 2.3 in
~\cite{BM15}. We also recall the proof for the reader's convenience.
\begin{lemma}\label{lem: F_p polynomial nonvanishing}
    There exists a universal constant $C>0$ such that if
    $f \in \mathbb{F}_p[T_1, \ldots, T_s]$ is a nonzero polynomial with $\deg(f) + 1 \leq d$, then there exists a maximal ideal $\mathfrak{q} \subset \mathbb{F}_p[T_1, \ldots, T_s]$ where
    $$
    f \neq 0 \text{ mod } \mathfrak{q},
    $$
    and such that
    $$
    |\mathbb{F}_p[T_1, \ldots, T_s]/\mathfrak{q}| \leq d^{C \log(p)}.
    $$
\end{lemma}
\begin{proof}
    Set $h \in \mathbb{F}_p[\tau]$ to be the nonzero trace polynomial of degree $\deg(h) = r \leq d^{2s+1}$ obtained from Lemma \ref{lem:reduction to one variable}. Let $I_m(p)$ be the number of monic irreducible polynomials in $\mathbb{F}_p[\tau]$ of degree $m$. A result of Gauss (see for instance \cite[Corollary 9.2.3]{roman2006field}) asserts
    $$
    I_m(p) = \frac{1}{m} \sum_{d \mid m} \mu(d) p^{m/d}
    $$
    where $\mu(d)$
    is the M\"obius function.
    For large values of $m$, we have
    $$
    \frac{1}{2m}p^m \leq I_m(p) \leq 2 \frac{1}{m}p^m,
    $$
    as follows from the classical Prime Polynomial Theorem.
    Therefore, $I_m(p) \geq p^{m/2}$ for large enough $m$.
    Since $\deg(h) \leq d^{2s+1},$ there exists an irreducible polynomial $w(\tau)$ of degree at most $C' \log(d)$ such that
    $w$ does not divide $h$, and where the constant $C'$ depends on $s$. To see this
    fact, we suppose the contrary and note that for a suitably chosen value of $C'$ depending only on $s$, the
    product of all distinct monic polynomials
    of degree at most $C'\log (d)$ would have degree larger than $d^{2s+1}$, a contradiction.
    
    We now see that
    $$
    |\mathbb{F}_p[\tau]/(w(\tau))| \leq p^{C'\log(d)}.
    $$
    We see that the map $\mathbb{F}_p[T_1, \ldots, T_s] \longrightarrow \mathbb{F}_p[\tau]$ given by evaluation of elements of $\mathbb{F}_p[T_1, \ldots, T_s]$ on the $s$-tuple $(\tau^{n_1}, \ldots, \tau^{n_s})$ is a ring homomorphism. Writing $\varphi$ for this ring homomorphism and $q$ for the quotient map $\mathbb{F}_p[\tau]\longrightarrow \mathbb{F}_p[\tau]/(w(\tau))$, we see that \[q \circ \varphi \colon \mathbb{F}_p[T_1, \ldots, T_s] \longrightarrow \mathbb{F}_p[\tau]/(w(\tau))\] is a surjective ring homomorphism onto a finite field. Its kernel $\mathfrak q$ is a maximal ideal, as desired.
\end{proof}

While the next two lemmas are known to experts, we include their proof for completeness and for the convenience of the reader.
\begin{lemma}\label{lem: degree bound char 0}
    Let $\mathbb{K} = \mathbb{Q}$ or $\mathbb{F}_p$, and suppose that $G \leq \GL_\ell(\mathbb{K}(T))$ is a finitely generated group, where here
    $T$ is a single indeterminate. Let $X$ be a finite generating set for $G$, and let $a= (a_{ij})$ be an
    element of $G$. If $\Phi$ is the product of all of the denominators of matrix coefficients of elements in $X$, then there exists a constant $K = K(X)$ such that
    $$
    \max\{\deg(\Phi(T)^{\|a\|_X}a_{ij}) \: : \: 1 \leq i,j \leq \ell\} \leq K\|a\|_X.
    $$
\end{lemma}
\begin{proof}
    Define $K = \max\{\deg(x_{ij}) \: : \: x = (x_{ij}), x \in X\}.$ There are finitely many polynomials that occur in the denominators of the coefficients of elements of $X$, and in particular, if $x = (x_{ij})$ for $x \in X$, we have $x_{ij} \in R[\frac{1}{S}][T]$ where $R$ is either $\mathbb{Z}$ or $\mathbb{F}_p$ and $S$ is a finite collection of elements in $R[T]$. Therefore, we may write
    $G \leq \GL_\ell(R[\frac{1}{S}][T])$. We then define
    \[
    K = \max\{\deg(\Phi(T)x_{ij}) \: : \: x= (x_{ij}), x \in X \}
    \]
    We proceed by induction on word length, and note that the statement is clear when $\|a\|_X = 1.$ Now assume that the statement is true for $n > 1$, and suppose that $\|a\|_X = n+1.$ We may write $a = bx$ where $\|b\|_X = n$ and $x \in X$. Letting $D = \Phi(T) \cdot \text{Id}_{\ell \times \ell}$, we then note $D^{n+1}a = (D^nb)(Dx)$
    because $D$ is central in $\GL_\ell(\mathbb{K}(T))$. By induction, we may write $D^nb = (\alpha_{ij})$ where $\deg(\alpha_{ij}) \leq K n$ for all
    $\{i,j\}$. We note that entries of $D^{n+1}a$ are scalar products of the rows of $D^nb$ and the columns of $Dx.$ We then write
    \begin{eqnarray*}
    \deg(\Phi^{n+1}a_{is}) &=& \deg \left(\sum_{j=1}^\ell \alpha_{ij}\cdot \Phi \cdot x_{js} \right) \\
    &\leq& \max\{ \deg(\alpha_{ij} \cdot \Phi \cdot x_{js}) \: : \: 1 \leq j \leq \ell \}\\
    &\leq& \max\{ \deg(\alpha_{ij}) + \deg(\Phi \cdot x_{js}) \: : \: 1 \leq j \leq \ell\}\\
    &\leq& K n + K\\
    &=& K(n+1),
    \end{eqnarray*}  
    as desired.
\end{proof}

\begin{lemma}\label{lem:magnitude polynomial coefficients}
    Suppose that $G \leq \GL_\ell(\mathbb{Q}(T))$ is a finitely generated group where $T$ is a single indeterminate. Let $X$ be a finite generating set for $G$, and let $a\in G$. Adopt the following notation:
    \begin{enumerate}
        \item Let $\Phi$ be the product of all of the denominators of matrix coefficients of elements in $X$;
        \item Write $x= (x_{ij}) \in X$;
        \item Write  \[\Phi(T)x_{ij} = \sum_{m=0}^{d_{ij}}\alpha_{ij,m}T^m\] for each pair of indices;
        \item Let $C = C(X) = \max_{i,j,m} \left\{|\alpha_{ij,m}|\right\}$;
        \item Let $\Phi(T)^{\|a\|_X}\cdot a = (a_{ij})$;
        \item Let $K=K(X)$ be the constant furnished by
        Lemma~\ref{lem: degree bound char 0}.
    \end{enumerate}
    If we write $a_{ij} = \sum_{m = 0}^{d_{ij}} \eta_{ij,m}T^m$,
    then
    $$
    \max\{|\eta_{ij,m}| \: : \: 1 \leq i,j \leq \ell\} \leq (2K\cdot C \cdot \ell)^{\|a\|_X}\cdot (\|a\|_X)!.
    $$
\end{lemma}
\begin{proof}
    Lemma \ref{lem: degree bound char 0} implies that the polynomials in the matrix coefficients of $\Phi^{\|a\|_X} \cdot a$ have degree bounded by $K\|a\|_X$. We proceed by induction on word length, and it is easy to see that the conclusion holds for the base case of words of length one.

    We proceed similarly to Lemma~\ref{lem: degree bound char 0}.
    Assume the conclusion holds when the word length is $n$, and we let $\|a\|_X = n+1$. We may write $a = bx$ where $\|b\|_X = n$ and $x \in X$. Letting $D = \Phi(T) \cdot \text{Id}_{\ell \times \ell}$, we have $D^{n+1}a = (D^na)(Dx)$ because $D$ is central in $\GL_\ell(\mathbb{Q}(T))$. We write $D^nb = (\beta_{ij})$ where $\beta_{ij} = \sum_{m=0}^{d_{ij}}\beta_{ij,m}T^m$, and by induction, we have $|\beta_{ij,m}| \leq (2K C\ell)^{n}n!$ for all $i,j,m$. Since entries of $D^{n+1}a$ are scalar products of the rows of $D^nb$ and the columns of $Dx,$ we then write
    \begin{eqnarray*}
        a_{is} &=& \sum_{j=1}^\ell \beta_{ij}\cdot D \cdot x_{js} \\
        &=& \sum_{j=1}^\ell \left( \sum_{m=0}^{d_{ij}}\beta_{ij,m}T^m \right)\left(\sum_{w=0}^{v_{ij}} \alpha_{js,w} T^w\right)\\
        &=& \sum_{j=1}^\ell \sum_{t=0}^{d_{ij} + v_{ij}}\sum_{m + w = t}\beta_{ij,m} \alpha_{js,w} T^t.
    \end{eqnarray*}
        Lemma \ref{lem: degree bound char 0} implies that $d_{ij} + v_{ij} \leq K(n+1)$. We now have an estimate on the absolute value of
        $\eta_{is,t}$ via:
        \begin{eqnarray*}
        \left|\sum_{j=1}^\ell\sum_{m + w = t}\beta_{ij,m} \alpha_{js,w}\right|&\leq& \sum_{j=1}^\ell\sum_{m+w=t}|\beta_{ij,m} \alpha_{js,w}|\\
        &\leq& \sum_{j=1}^\ell \sum_{m+w=t}C\cdot (2K C\ell)^{n}n! \\
        &\leq& 2\ell\cdot C\cdot K(n+1)\cdot (2K C\ell)^{n}n!=(2K C\ell)^{n+1}(n+1)!,
        \end{eqnarray*}
        as desired.
\end{proof}

\section{More on finite quotients of malabelian groups}\label{s:malabelian}
In this section, we revisit the functions $\RF_{G,\mathcal{F}^A}(n)$ for when $G$ is a finitely generated uniformly malabelian group. We then develop the necessary tools to show the forward direction of Theorem \ref{thm:main-growth}. In particular, we show that if $\RF_{G,\mathcal{F}^A}(n) \preceq n^d$ for some natural number, then $G$ admits a faithful finite dimensional representation over some field when $G$ is a uniformly malabelian group.

\subsection{Finite quotients of infinite groups}
The reader will recall the discussion of residual finiteness growth from 
the introduction.

Let $\mathcal F$ denote a family of
finite products of nonabelian finite simple groups
and let $\mathcal H$ denote powers of nonabelian finite simple groups which
occur as factors of elements of $\mathcal F$.
The following lemma says that when $G$ is residually-$\mathcal{F}^A$, then $G$ is residually-$\mathcal{H}^A$, where here $\mathcal{H} = \{S_i^{\ell_i}\}_{i \in \mathbb{N}}$ where each $S_i^{\ell_i}$ is a factor of $G_{n_i}\in\mathcal F$ for some $n_i$ for all $i$. Moreover, we have control over the residual
 finiteness growth functions:
$$
\RF_{G, \mathcal{H}^A}(n) \preceq \RF_{G, \mathcal{F}^A}(n).
$$
\begin{lemma}\label{lem:residual F reduction} Let $G$ be a finitely generated center-free group with a finitely generated group $A \leq \Out(G)$. We let:
\begin{itemize}
    \item $\mathcal{F}$ be a collection of finite products of nonabelian finite simple groups.
    \item $\mathcal{H}$ be the collection of finite products of finite simple groups of the form $S^{\ell}$, where $S$ is simple and $S^{\ell}$ appears
    as a factor of some member of $\mathcal F$.
    \end{itemize}
    
If $G$ is residually-$\mathcal{F}^A$, then $G$ is residually-$\mathcal{H}^A$.
Moreover,
$$
\RF_{G, \mathcal{H}^A}(n) \preceq \RF_{G, \mathcal{F}^A}(n).
$$
\end{lemma}
\begin{proof}
    Throughout, we fix a finite generating set $X$ for
    $G$.
    Let $x \in G$ be a nontrivial element of length at most $n$. By
    assumption, there exists an epimorphism $\varphi \colon G \longrightarrow Q$ with $\Gamma_{G,A}$-invariant kernel where $Q \in \mathcal{F}$ such that $\varphi(x) \neq 1$ and
    $$
    |Q| \leq \RF_{G,\mathcal{F}^A}(n).
    $$
    We may write $Q = \prod_{i=1}^\ell Q_i^{s_i}$ where $\{Q_i\}_{1\leq i\leq
    \ell}$ are distinct nonabelian finite simple groups. For each $1 \leq j \leq \ell$, we let \[q_j \colon \prod_{i=1}^\ell Q_i^{s_i} \longrightarrow Q_j^{s_j}\] be the natural projection. It is immediate
    that $q_j \circ \varphi$ has a $\Gamma_{G,A}$-invariant kernel for all $1 \leq j \leq \ell$, and given that $\varphi(x) \neq 1$, there exists $1 \leq j_0 \leq \ell$ such that $q_{j_0} \circ \varphi(x) \neq 1.$ We note that $Q_{j_0}^{s_{j_0}} \in \mathcal{H}$ by definition,
    and consequently $\D_{G, \mathcal{H}^A}(x) \leq \RF_{G, \mathcal{F}^A,X}(n)$. We thus obtain
    $$
    \RF_{G, \mathcal{H}^A,X}(n) \preceq \RF_{G, \mathcal{F}^A,X}(n),
    $$
    as desired.
\end{proof}

\subsection{Least common multiples in malabelian groups}
For a more detailed discussion of the following topics, including proofs of the many of the statements, see \cite[Section 3]{bou_rabee-McReynolds2011}. As usual, we let $G$ be a malabelian group.

Given a finite subset $T \subset G \backslash \{1\},$ we define \[H_T = \bigcap_{x \in T} \overline{\left< x\right>},\] where here $\overline{\left<g\right>}$ denotes the normal closure of the cyclic subgroup $\left<x\right>$. We call any nontrivial element in $H_T$ a \emph{common multiple of $T$ in $G$.} The following lemma can be found in \cite[Lemma 3.1]{bou_rabee-McReynolds2011}. The proof is very easy and we omit it; the commutator construction in the proof of Lemma~\ref{lem: length lcm} below gives a related explicit construction of common multiples.

\begin{lemma}\label{lem:LCM_nonvanishing}
    Let $G$ be a group, $T \subset G \backslash \{1\}$ be a finite subset, and $h$ a common multiple for $T$ in $G$. If $\varphi \colon G \longrightarrow H$ is a homomorphism such that $\varphi(h) \neq 1$, then $\varphi(t) \neq 1$ for all $t \in T.$
\end{lemma}

Nontrivial common multiples always exist in malabelian groups, and the proof
of the following lemma is also easy, and proceeds
by induction on the size of $T$:

\begin{lemma}\label{lem:LCM exists}
    If $G$ is a malabelian group and $T \subset G \backslash \{1\}$ is a finite subset, then $H_T$ is nontrivial and $T$ has a common multiple.
\end{lemma}

We say that $G$ is \emph{fully residually-$\mathcal{F}^A$} if, for every finite subset $T\subset G\backslash\{1\}$, there is an epimorphism from $G$ to a member of $\mathcal F$, with $\Gamma_{G,A}$--invariant kernel, which is nontrivial on every element of $T$. The existence of a common multiple for any finite subset of nontrivial elements of a malabelian group $G$ immediately implies that if $G$ is residually-$\mathcal{F}^A$ for some family of finite groups $\mathcal{F}$ and $A \leq \Out(G)$ is finitely generated, then $G$ must also be fully residually-$\mathcal{F}^A$:
\begin{lemma}\label{lem:malabelian residually F implies fully residually F}
    Let $G$ be a malabelian group, and suppose that
    $A \leq \Out(G)$. If $G$ is residually-$\mathcal{F}^A$ then $G$ is fully residually-$\mathcal{F}^A.$
\end{lemma}

For the remainder of this section, we will assume that $G$ is
finitely generated and uniformly malabelian.
For a finite subset $T \subset G \backslash \{1\}$, we define the \emph{least common multiple length} of $T$ relative to $X$ to be \[\lcm_X(T) = \min\{ \|a\|_X : a \in H_T \backslash \{1\}\}.\] Any element $x \in H_T$ where $\|x\|_X = \lcm_X(T)$ is a \emph{least common multiple} for the subset $T.$

The next lemma estimates an upper bound for the length of a least common multiple for a finite subset $T$ in a finitely generated uniformly malabelian group in terms of the lengths of elements in $T$ and the size of $T$.
\begin{lemma}\label{lem: length lcm}
    Let $G$ be a finitely generated, uniformly malabelian group with a finite generating set $X$, and let $\kappa$ be a uniformly malabelian constant of $G$ with respect to $X$.
    If $T \subset G \backslash \{1\}$ is a finite subset, then 
    \[
    \lcm_X(T) \leq 4 |T|^2 (\max\{\|a\|_X \: : \: a \in T\}+3\kappa).
    \]
\end{lemma}
\begin{proof}
    Let $d = \max\{\|a\|_X \: : \: a \in T\}$. Let $T = \{x_1, \ldots, x_\ell\}$, and let $k$ be the smallest number such that $2^{k-1}< \ell \leq 2^k.$ We add to the set $\{x_1, \ldots, x_\ell \}$ enough elements such that the new set has $2^k$ elements, which we write $\{x_1, \ldots, x_{2^k}\}$. Note that this list may contain repetitions.
    
    For each pair $x_{2i-1}$ and $x_{2i}$, we replace $x_{2i}$ by $y_i x_{2i} y_i^{-1}$ for some $\|y_i\|_X \leq \kappa$ with \[[x_{2i-1}, y_i x_{2i} y_i^{-1}] \neq 1.\]
    We now define a new set of elements $\{x_i^{(1)}\}_{i=1}^{2^{k-1}}$ by the
    rule $x_i^{(1)} = [x_{2i-1}, x_{2i}]$, and observe that $\|x_i^{(1)}\|_X \leq 4(d+2\kappa).$ We now have $2^{k-1}$ elements in this set, and we then repeat the above process again by replacing $x_{2i}^{(1)}$ with a conjugate if necessary (at the expense of increasing the length by at most $2\kappa$), in order to ensure that
    $x_{2i-1}^{(1)}$ and $x_{2i}^{(1)}$ do not commute. Setting $x_i^{(2)} = [x_{2i-1}^{(1)}, x_{2i}^{(1)}],$ we obtain $2^{k-2}$ nontrivial elements 
    $\{x_i^{(2)}\}_{i=1}^{2^{k-2}}$,
    with \[\|x_i^{(2)}\|_X \leq 4(4(d+2\kappa) + 2\kappa).\]
    
    Repeating this process, $k\geq 2$ times, we obtain an element
    $x_i^{(k)}  \in H_T$ such that $\|x_i^{(k)}\|_X \leq 4^k d + a_k$ where $a_k$ is defined inductively $a_1 = 8\kappa$ and $a_j =  4(a_{j-1} + 2\kappa).$ By induction, we see that
    \[
a_j = 2\kappa \cdot\sum_{\ell=1}^j 4^{\ell}.
    \]
    Since $4^k \leq 4 \ell^2$, we have
    \[
    \|x_1^{(k)}\|_X \leq 4^k \cdot d + a_k = 4^k\cdot d +\frac{8\kappa}{3}(4^k-1) \leq 4^k(d+3\kappa) \leq 4\ell^2 (d+3\kappa).
    \]
    Since $\lcm_X(T) \leq \|x_1^{(k)}\|_X$, we obtain the desired estimate.
\end{proof}

\section{Residual finiteness growth and linearity}\label{s:rf-lin}

In this section, we will prove the main general results of this
paper concerning residual finiteness growth and linearity.

\subsection{Growth to linearity}

Before we prove the forward direction of Theorem \ref{thm:main-growth}, we have the following simple lemma, whose proof is easy and we omit.
\begin{lemma}\label{lem:bound rank A to bound rank Aut(A)}
    Let $G$ be a finitely generated center-free group, and suppose that $A \leq \Out(G)$ is a finitely generated group. Suppose that
    $\mathcal{F}$ is a family of groups
    such that $G$ is residually-$\mathcal{F}^A$. Then $\Gamma_{G,A}$ is residually-$\mathcal{H}$, where $\mathcal{H}$ consists of
    automorphism groups of elements of $\mF$.
\end{lemma}

Now, let
$\mF$ be a family of finite products of nonabelian finite simple groups.
We say that $\mF$ is \emph{factor-closed} if whenever $H_1$ and $H_2$
are finite products of finite nonabelian simple groups such that $H_1\times H_2\in\mF$, then $H_1,H_2\in\mF$.
We now prove the forward direction of Theorem \ref{thm:main-growth}.
\begin{prop}\label{prop:polynomial implies bounded rank}
    Let $G$ be a finitely generated uniformly malabelian group with an infinite order element $a_0$, and suppose that $A \leq \Out(G)$ is a finitely generated group. Let $\mathcal{F}$ be a
    factor-closed
    set of finite products of nonabelian finite simple groups of Lie type
    that is $e$--extension-bounded for some $e\in\N$.
    
    If 
    \[
    \RF_{G,\mathcal{F}^A}(n) \preceq n^d
    \]
    for some $d \in \mathbb{N}$, then there exists an $R>0$ and
    an $e$--extension-bounded family of finite products of nonabelian finite simple groups of bounded multiplicity $\mathcal{H}\subseteq\mathcal F$
    such that
    $G$ is residually-$\mathcal{H}^A$, and such that
    the rank of $\Aut(H)$ is bounded above by $R$ for all $H\in\mathcal H$.
\end{prop}
\begin{proof}
From Lemma \ref{lem:residual F reduction}, we may assume that $\mathcal{F}$ consists of groups of the form $H_i^{\ell_i}$, with $H_i$ a nonabelian finite simple group of Lie type occurring as a factor of an
element of $\mF$. Let $X$ be a finite generating set for $G$.
    
    Choose $C_1$ to be a uniformly malabelian constant for $G$ with respect to $X$. We will show that there exists a subcollection $\mathcal{H}$ of $\mathcal{F}$ consisting of groups of rank bounded by $R$ for some constant $R>0$, such that $G$ is residually-$\mathcal{H}^A$.

    Let $a\in G$ be nontrivial.
    Since $G$ is uniformly $C_1$-malabelian, there exists an element $b_0 \in G \backslash \{1\}$ such that $[b_0 ab_0^{-1}, a_0] \neq 1$ with $\|b_0\|_X \leq C_1.$ Let 
    \[
    T_{a,n} = \{[b_0 ab_0^{-1}, a_0] , a_0^2, \ldots, a_0^n\};
    \] here the reader may treat $n$ as a variable to be fixed later.
    Since 
    $$
    \|[b_0 ab_0^{-1},a_0]\|_X \leq 4C_1 + 2\|a\|_X + \|a_0\|_X,
    $$
    we see that if 
    $$
    n \geq n(a) =  8\max\{C_1, \|a\|_X, \|a_0\|_X\},
    $$ then $\|t\|_X \leq n \|a_0\|_X$ for all $t \in T_{a,n}$. Lemma \ref{lem: length lcm} implies  that if $k_{a}$ is a least common multiple of $T_{a,n(a)}$, then
    \[
\|k_{a}\|_X \leq 4n(a)^2(n(a)\|a_0\|_X+3C_1)\leq C_2 (n(a))^3
    \]
    where $C_2 = C_2(X)$ is chosen suitably.

    By assumption, there exists a constant $C_3 = C_3(X)$ for which there is a power of a nonabelian finite simple group $H_{a}^{\ell_{a}} \in \mathcal{F}$ and an epimorphism $\varphi_{a} \colon G \longrightarrow H_{a}^{\ell_{a}}$ with $\Gamma_{G,A}$-invariant kernel such that $\varphi_{a}(k_{a}) \neq 1$, satisfying
    \[
    |H_{a}^{\ell_{a}}| \leq C_3 (\|k_{a}\|_X)^d \leq C_2^d \: C_3 \: (n(a))^{3d} = C_4 \: (n(a))^{3d},
    \]
    where here $C_4 = C_4(X) = C_2^d \: C_3$. We fix such a $\varphi_a$
    for each nontrivial $a\in G$ for the remainder of the proof,
    and we let $\mathcal H$ consist of the groups $H_a^{\ell_a}$.
    
    Since $\varphi_{a}(k_{a}) \neq 1$, Lemma \ref{lem:LCM_nonvanishing} implies that $\varphi_{a}(a_0^j) \neq 1$ for all $1 \leq j \leq n(a).$ Hence, we have the \emph{a priori} estimate 
    on the size of the cyclic group generated by $\varphi_a(a_0)$ given by $|\left<\varphi_{a}(a_0)\right>| \geq n(a)$, 
    whence it follows that $m_1(H_{a}^{\ell_{a}}) \geq n(a)$. Therefore, 
    \[
\frac{\log |H_{a}^{\ell_{a}}|}{\log (m_1(H_{a}^{\ell_{a}}))} \leq \frac{\log (C_4 \: (n(a))^{3d})}{\log (n(a))} = \frac{C_4}{\log (n(a))} + 3d \frac{\log (n(a))}{\log (n(a))} = 3d + \frac{C_4}{\log (n(a))}.
    \]
    Thus, the set
    \[
\left\{ \frac{\log |H_{a}^{\ell_{a}}|}{\log (m_1(H_{a}^{\ell_{a}}))} \right\}_{a \in G \backslash \{1\}}
    \]
   is bounded by some constant $C_5 = C_5(X)$.  

   It suffices to show that the set of exponents
   $\{\ell_a\}_{a \in G \backslash \{1\}}$,
   coming from the targets of the maps $\{\varphi_a\}_{a\in G}$,
   is bounded.
   To this end, we show that the inequality \[(n(a))^{\ell_a} \leq |H_a^{\ell_a}| \leq C_4 \: (n(a))^{3d}\] holds for all $a \in G \backslash \{1\}$. Since $\varphi_a(k_{a}) \neq 1$, we may write
   its image as a tuple
   \[\varphi_a(k_{a}) = (\alpha_{i})_{i=1}^{\ell_a}\in H_a^{\ell_a},\] where $\alpha_{i_0} \neq 1$ for some $1 \leq i_0 \leq \ell_a.$ In particular, if $\lambda \colon H_a^{\ell_{a}} \longrightarrow H_a$ is the projection onto the $i_0^{th}$ factor, then $\lambda \circ \varphi_a(k_{a}) \neq 1$. Hence, Lemma \ref{lem:LCM_nonvanishing} implies $\lambda \circ \varphi_a(a_0^j) \neq 1$ for $1 \leq j \leq n(a).$ Therefore, \[n(a) \leq |\left<\lambda \circ \varphi_a(k_{a})\right>| \leq |H_a|.\]
   Raising to the $\ell_a$-th power, we see that
   \[
   (n(a))^{\ell_a} \leq |H_a|^{\ell_a}= |H_a^{\ell_a}| \leq C_4 \: (n(a))^{3d}.
   \]
   Hence, 
   \[
   \ell_a\log (n(a)) \leq \log C_4+3d\log(n(a)),
   \]
   and so $\ell_a\leq 3d+C_6$ for a suitable constant $C_6$ that is
   independent of $a$. Since this inequality holds for
   all $a \in G \backslash \{1\}$, we see that the set $\{\ell_a\}_{a \in G \backslash\{1\}}$ is bounded by a constant $C_7 = C_7(X)$.
   It follows that $\mathcal H$ has bounded multiplicity. That the
   ranks of automorphism groups of elements of $\mathcal H$ are universally bounded follows
   from the fact that each element of $\mathcal H$ is $e$--extension-bounded,
   and from Lemma~\ref{lem:rank_inequality1}.
   \end{proof}

Thus we obtain:
\begin{cor}\label{cor:poly imply linear}
     Let $G$ be a finitely generated uniformly malabelian group with an infinite order element, and suppose that $A \leq \Out(G)$ is a finitely generated group.
     Let $\mathcal{F}$ be a set of finite products of nonabelian finite simple groups of Lie type that are $e$--extension-bounded for some
     $e\in \N$. If 
     \[
     \RF_{G, \mathcal{F}^A}(n) \preceq n^d
     \] 
     where $d\in \mathbb{N}$, then there exists an injective homomorphism $\varphi \colon \Gamma_{G,A} \longrightarrow \GL_\ell(\mathbb{K})$ for some field $\mathbb{K}$ and $\ell \in \mathbb{N}$.
\end{cor}
\begin{proof}
    Clearly we may assume that $\mathcal F$ is factor-closed.
    By Proposition~\ref{prop:polynomial implies bounded rank}, we
    have that $G$ is residually $\mathcal H^A$, where $\mathcal H
    \subseteq\mF$ consists of powers of finite simple groups of Lie type
    of the form $H^{\ell}$, and so that:
    \begin{enumerate}
        \item there is a universal bound on the multiplicity
        for all elements
        of $\mathcal H$;
        \item there is a universal bound on the rank of the
        automorphism group of each element of $\mathcal H$.
    \end{enumerate}
    By Lemma~\ref{lem:bound rank A to bound rank Aut(A)},
    we have that $\Gamma_{G,A}$ is residually $\mathcal A$, where
    $\mathcal A$ consists of automorphism groups of elements of
    $\mathcal H$.
    We obtain a faithful linear representation of $\Gamma_{G,A}$
    immediately from Lemma \ref{lem:ultraproduct representation}. 
\end{proof}

\subsection{Linearity to growth}

In this section, we let $\mathcal F$ denote finite products of finite
simple groups of Lie type. If $e\in\N$, we write $\mathcal F_e\subseteq
\mathcal F$ for the elements of $\mathcal F$ which are $e$--extension-bounded.

\begin{thm}\label{thm:rep_poly_growth}
    Let $G$ be a finitely generated  uniformly malabelian group, and suppose that $A \leq \Out(G)$ is a finitely generated subgroup.
Suppose that $\Gamma_{G,A}$ has a faithful representation \[\varphi \colon \Gamma_{G,A} \longrightarrow \GL_\ell(\mathbb{K})\] for some field
$\mathbb K$.
    Then there exists a finite index characteristic subgroup $G_\ell \trianglelefteq G$ and a natural number $d$ such that
    \[
    \RF_{G_\ell, \mathcal{F}^{\Gamma_{G,A}/G_\ell}}(n) \preceq n^{d}.
    \]
    Moreover, if $\mathbb K$ has characteristic zero then there is an
    $e\in\N$ such that 
    \[
    \RF_{G_\ell, \mathcal{F}_e^{\Gamma_{G,A}/G_\ell}}(n) \preceq n^{d}.
    \]
\end{thm}
\begin{proof}
     Let $G_{\ell}$ be the intersection of all finite index subgroups
     of $G$ of index at most $J(\ell)$; see Theorem~\ref{thm:larsen-pink}. 
    Let $X$ be a finite generating set for $\Gamma_{G,A}$ which includes a finite generating set $Y$ for $G_\ell$ and a finite generating set $Z$ for $G$; thus we have inclusions $Y\subseteq Z\subseteq X$. 
    
    By Lemma~\ref{lem:polynomial-localization}, taking $\mathfrak R = \mathbb{Z}[T_1, \ldots, T_s]$ or $\mathbb{F}_p[T_1, \ldots, T_s]$ and $R\in\{\Z,\mathbb F_p\}$ depending on the characteristic of the defining field, there exists a finite subset $S \subset \mathfrak R$ consisting
    of nonzero elements
    such that \[\Gamma_{G,A} \leq \GL_{\ell}\left(R\left[\frac{1}{S}\right]\left[T_1, \ldots, T_s\right]\right).\] 
    
    Suppose first that \[\Gamma_{G,A} \leq \GL_\ell\left(\mathbb{Z}\left[\frac{1}{S}\right][T_1,\ldots
    T_s]\right).\] Let $\Phi$ be the product of all of the denominators of matrix coefficients of elements in $X$. Write $D = \Phi \cdot \text{Id}_{\ell \times \ell}$, and let $a \in G_\ell$ be a nontrivial element. Let $\kappa  = \kappa(Z)$ be the uniformly malabelian constant of $G$ with respect to $Z$.
    
    Lemma \ref{lem:bound length of element in D^n(G)} and
    Proposition~\ref{prop:lin-solvable} together imply there exists
    a universal constant $C_2$ and an element $h \in D^{C_1 \lceil\log(\ell)\rceil  +1}(G)$ satisfying
        \begin{enumerate}
    \item $\|h\|_Z \leq 8^{C_1  \log(\ell) + 1}\max\{\|a\|_Z, \kappa\}$;
    \item If $\varphi \colon G \longrightarrow Q$ is an epimorphism where $\varphi(h) \neq 1$, then $\varphi(a) \neq 1$;
    \item If $\varphi \colon G \longrightarrow Q$ is an epimorphism and $N$ is a normal subgroup of $Q$ such that
    $\varphi(a) \in N,$ then $\varphi(h) \in D^{C_1  \lceil\log(\ell)\rceil + 1}(N)$.
    \end{enumerate}
    Moreover, there is a constant $C_2  > 0$ such that $\|h\|_X \leq C_2\|a\|_Z$. Writing $h = (h_{ij})$ as a matrix,
    Lemma \ref{lem: degree bound char 0} implies that there exists a constant $K = K(X)$ such that
    $$
    \max\{\deg(\Phi^{\|h\|_X} h_{ij}) \: : \: 1 \leq i,j \leq \ell\} \leq K C_2\|a\|_Z.
    $$
    Thus,
    $$
    \max\{ \deg(\Phi^{\|h\|_X} h_{ij} - \Phi^{\|h\|_X}\delta_{ij}) \: : \: 1 \leq i,j \leq \ell\} \leq KC_2 \|a\|_Z,
    $$
    where here $\delta_{ij}$ denotes the Kronecker delta function.
    
    Since $h \neq \text{Id}_{\ell \times \ell},$ there exist indices $i_0$ and $j_0$ such that 
    \[
    f =\Phi^{\|h\|_X} h_{i_0j_0} - \Phi^{\|h\|_X}\delta_{i_0j_0} \neq 0.
    \]
    Lemma \ref{lem:reduction to one variable} implies the existence of a
    sequence of natural numbers $(n_1, \ldots, n_s)$ contained in
    $\{0, 1, \ldots,  (KC_2\|a\|_Z)^{2s} \}$ such that if $\tau$ is an indeterminate, then $g(\tau) = f(\tau^{n_1}, \ldots, \tau^{n_s}) \neq 0$, and
    $\deg(g) \leq (KC_2 \|a\|_Z)^{2s+1}$.

    Viewing $\Phi$ as a function of $\{T_1,\ldots,T_s\}$, we note that if
    $\Phi(\tau^{n_1}, \ldots, \tau^{n_s})$ vanishes identically then
    $f$ also vanishes identically. It follows that $\Phi$ does not vanish under
    the substitution of powers of $\tau$, and so neither can the
    denominators of any of the matrix entries in $X$.

    It follows that the evaluation map \[\psi\colon \Z[T_1,\ldots, T_s]
    \longrightarrow \Z[\tau]\] defined by 
    $$
    \psi(w[T_1, \ldots, T_s]) = w[\tau^{n_1}, \ldots, \tau^{n_s}]
    $$
    sends elements of $S$ to a collection $S'$ of 
    nonzero elements in the target, whence
    one obtains a well-defined extended evaluation map
    \[\psi \colon \Z\left[\frac{1}{S}\right][T_1, \ldots, T_s] \longrightarrow \Z\left[\frac{1}{S'}\right][\tau]\] and a
    group homomorphism
    \[\bar{\psi} \colon \GL_\ell\left(\mathbb{Z}\left[\frac{1}{S}\right][T_1, \ldots, T_s]\right) \longrightarrow \GL_\ell\left(\mathbb{Z}\left[\frac{1}{S'}\right][\tau]\right).\] 
    In particular, we have $\bar{\psi}(h) \neq 1$ since $\psi(g) \neq 1.$ Additionally, we see that $\|\bar{\psi}(h)\|_{\bar{\psi}(X)} \leq \|h\|_X \leq C_2\|a\|_Z.$
    
    Fix an arbitrary bound on the coefficients of $\Phi$ (which depends only
    on $X$), and consider
    a substitution map of the form $w(T_1, \ldots, T_s) \longrightarrow w(\tau^{n_1}, \ldots, \tau^{n_s})$. Notice that the coefficients of  $\bar{\psi}(\Phi)$ will be bounded by a constant $C_3$ that
    depends only on the bounds of the coefficients of $\Phi$ and on
    $s$. Writing \[g(\tau) = a_0 + a_1\tau + \cdots +a_d \tau^d\] with
    the bound $d \leq (KC_2\|a\|_Y)^{2s+1}$, Lemma~\ref{lem: degree bound char 0} and Lemma \ref{lem:magnitude polynomial coefficients} imply
    the existence of a constant $K'$ such that
    $$
    |a_i| \leq (2K'\cdot C_3 \cdot \ell)^{KC_2\|a\|_Y}\cdot (\|a\|_Y)!.
    $$
    Lemma \ref{lem: integer polynomial nonvanishing} implies that
    there exists an integer $0 \leq t \leq (KC_2\|a\|_Y)^{2s+1} + 1$ and a prime $p$ such that
    $$
    g(t) \neq 0 \pmod p,
    $$
    and such that 
    \begin{eqnarray*}
    p &\leq& C_4\left(\log((2K'\cdot C_3 \cdot \ell)^{KC_2 \|a\|_Y}\cdot(\|a\|_Y)!)+(2s+2)(KC_2 \|a\|_Y)^{(2s+1)(2s+2)}\right)\\
    &\leq& C_4\left(  \left(KC_2 \|a\|_Y\right) (\log(K' \cdot C_3 \cdot \ell)\cdot \log((\|a\|_Y)!) + (2s+2)(KC_2\|a\|_Y)^{(2s+1)(2s+2)} \right);
    \end{eqnarray*}
    here, the constant $C_4 = C_4(s)$ depends on $s$ alone. Since (up to 
    a multiplicative constant) we have
    \[\log((\|a\|_Y)!)\leq \|a\|_Y\cdot\log(\|a\|_Y)\leq (\|a\|_Y)^2,\] we see that
    there exists a natural number $M$ and a constant $C_5 = C_5(X)$ such that
    $$
    p \leq C_5(\|a\|_Y)^M.
    $$
    Observe that if $\bar{\psi}(\Phi)(t) = 0\pmod p$, then
    \begin{eqnarray*}
    g(t) &=& \bar{\psi}(\Phi^{\|h\|_X} h_{i_0j_0} - \Phi^{\|h\|_X}\delta_{i_0j_0})(t)\pmod p\\
    &=& \bar{\psi}(\Phi^{\|h\|_X})(t) \cdot \bar{\psi}(h_{i_0j_0} - \delta_{i_0j_0})(t)\pmod p\\
    &=& 0\pmod p,
    \end{eqnarray*}
    which is a contradiction. In particular, the polynomial
    $\bar{\psi}(\Phi)(\tau)$ is nonzero modulo $p$.
    
    Hence, the ring map $\lambda \colon \Z[\tau] \longrightarrow \mathbb{F}_p$ given by $\lambda(w) = w(t) \pmod p$ is well defined and has the property
    that $\lambda(s)\neq 0$ for all $s\in S'$; in particular $\lambda$
    extends to a ring homomorphism \[\lambda \colon \Z\left[\frac{1}{S'}\right][\tau] \longrightarrow \mathbb{F}_p,\] and induces a homomorphism of
    general linear groups \[\bar{\lambda} \colon \GL_\ell\left(\mathbb{Z}\left[\frac{1}{S'}\right][\tau]\right) \longrightarrow \GL_\ell(p).\] Thus, we have an induced map $(\bar{\lambda} \circ \bar{\psi})|_{\Gamma_{G,A}} \colon \Gamma_{G,A} \longrightarrow \GL_\ell(p)$, for which the subgroup 
    $$
    (\ker(\bar{\lambda} \circ \bar{\psi}) \cap \Gamma_{G,A})
    $$
    is a normal subgroup of $\Gamma_{G,A}$ not containing the element
    $h$. Thus, 
    $$
    \ker((\bar{\lambda} \circ \bar{\psi})|_{G_\ell}) = G_\ell \cap (\ker(\bar{\lambda} \circ \bar{\psi}) \cap \Gamma_{G,A})
    $$
    is $\Gamma_{G,A}$-invariant since both $G_\ell$ and $(\ker(\bar{\lambda} \circ \bar{\psi}) \cap \Gamma_{G,A})$ are $\Gamma_{G,A}$-invariant. Letting $(Q_1, Q_2, Q_3)$ be a Larsen-Pink triple for $Q = \bar{\lambda} \circ \bar{\psi}(G)$, we see that $\bar{\lambda} \circ \bar{\psi}(G_{\ell})\leq Q_1$. To see this, note that $Q/Q_1$ has order at most
    $J(\ell)$ by the definition of a Larsen--Pink triple. Since $G_{\ell}$ is defined as the intersection of all
    subgroups of $G$ of index at most $J(\ell)$, we have that $G_{\ell}$
    is contained in the kernel of the composition \[G\longrightarrow
    Q\longrightarrow Q/Q_1.\]
    
    Moreover,
    $\bar{\lambda} \circ \bar{\psi}(h)$ is nontrivial, so that
    $\bar{\lambda} \circ \bar{\psi}(a)\notin Q_2$; thus 
    $q \circ \bar{\lambda} \circ \bar{\psi}(a) \neq 1$, where here $q \colon Q_1 \longrightarrow Q_1/Q_2$ is the natural projection. By construction, we have $Q_1/Q_2$ is a nontrivial product of nonabelian finite simple groups in characteristic $p$. We observe that
    $$
    \ker((\bar{\lambda} \circ \bar{\psi})|_{G_\ell}) \leq \ker(q \circ (\bar{\lambda} \circ \bar{\psi})|_{G_\ell}).
    $$
    Since $\ker((\bar{\lambda} \circ \bar{\psi})|_{G_\ell})$ is invariant under the conjugation action of $\Gamma_{G, A}$, we have \[\ker((\bar{\lambda} \circ \bar{\psi})|_{G_\ell}) \leq g^{-1}(\ker(q \circ (\bar{\lambda} \circ \bar{\psi})|_{G_\ell}))g,\] where here $g\in \Gamma_{G,A}$ is arbitrary. Therefore,
     $$
    \ker((\bar{\lambda} \circ \bar{\psi})|_{G_\ell}) \leq \bigcap_{g \in \Gamma_{G,A}}g^{-1}(\ker(q \circ (\bar{\lambda} \circ \bar{\psi})|_{G_\ell}))g = (\ker(q \circ (\bar{\lambda} \circ \bar{\psi})|_{G_\ell}))_A.
    $$
    Finally, we see that 
    $$
    \left|G_\ell / (\ker(q \circ (\bar{\lambda} \circ \bar{\psi})|_{G_\ell}))_A\right| \leq p^{\ell^2} \leq C_5^{\ell^2}(\|a\|_Y)^{\ell^2M},
    $$
    as desired.

    For the positive characteristic case, we proceed in the same way, using
    Lemma \ref{lem: F_p polynomial nonvanishing} instead of Lemma \ref{lem: integer polynomial nonvanishing} and the corresponding positive-characteristic version of Lemma \ref{lem:magnitude polynomial coefficients}.

    In the case of characteristic zero, the semisimple-type quotients
    we obtain are $e$--extension-bounded for some $e$ depending only
    on $\ell$, by Corollary~\ref{cor:lin-extension-bound}.
\end{proof}

Combining Theorem~\ref{thm:rep_poly_growth} and Proposition~\ref{prop:polynomial implies bounded rank}, we obtain
Theorem~\ref{thm:main-growth}.

\section*{Acknowledgements}
The authors thank Ian Agol,
Ian Biringer, Tara Brendle, Emmanuel Breuillard, Martin Bridson, Jack Button, Asaf Hadari, Scott Harper, Faye Jackson, Dawid Kielak, Antonio L\'{o}pez Neumann, Dan Margalit, Curt McMullen, Ben McReynolds, Andrei Rapinchuk, and Andreas Thom   for helpful conversations and email correspondences. The authors thank an anonymous referee for helpful comments.

The first author was partially
supported by NSF grants DMS-2002596 and DMS-2349814, and by
Simons Foundation International Grant SFI-MPS-SFM-00005890 while this
research was carried out. The second author is supported by National Science Center Grant Maestro-13 UMO-2021/42/A/ST1/00306, and was supported
by a postdoctoral fellowship under NSF RTG grant DMS-1839968.

\section*{Conflict of Interest Statement}
On behalf of all authors, the corresponding author states that there is no conflict of interest.

\bibliography{bibs}
\bibliographystyle{plain}

\end{document}